\documentclass[12pt, reqno]{amsart}

\usepackage{amsthm}
\usepackage{amssymb, amsmath}

\usepackage{amscd}   
\usepackage[all]{xy} 
\usepackage{colonequals}
\usepackage{tikz-cd}
\usepackage{faktor}
\usepackage{soul}

\usepackage[backref, colorlinks=true, linkcolor=blue, citecolor=blue]{hyperref}
\usepackage[alphabetic,backrefs,lite]{amsrefs}

\usepackage[]{xcolor}

\usepackage[top=1in, bottom =1in, left=1in, right = 1in]{geometry}

\usepackage{comment}
\DeclareFontEncoding{OT2}{}{} 

\usepackage[hang,flushmargin]{footmisc}

\makeatletter
\renewcommand*\env@matrix[1][*\c@MaxMatrixCols c]{%
  \hskip -\arraycolsep
  \let\@ifnextchar\new@ifnextchar
  \array{#1}}
\makeatother

\newtheorem{lemma}{Lemma}[section]
\newtheorem{theorem}[lemma]{Theorem}

\newtheorem{prop}[lemma]{Proposition}
\newtheorem{cor}[lemma]{Corollary}

\newtheorem{claim*}{Claim}
\newtheorem{thm}[lemma]{Theorem}
\newtheorem{defn}[lemma]{Definition}

\theoremstyle{remark}
\newtheorem{remark}[lemma]{Remark}
\newtheorem{remarks}[lemma]{Remarks}

\newcommand{\A}{{\mathbb A}}

\newcommand{\PP}{{\mathbb P}}

\newcommand{\F}{{\mathbb F}}

\newcommand{\Q}{{\mathbb Q}}
\newcommand{\RR}{{\mathbb R}}
\newcommand{\Z}{{\mathbb Z}}

\newcommand{\kbar}{{\overline{k}}}
\newcommand{\Kbar}{{\overline{K}}}

\newcommand{\kk}{{\mathbf k}}

\newcommand{\calB}{{\mathcal B}}
\newcommand{\calC}{{\mathcal C}}
\newcommand{\calD}{{\mathcal D}}

\newcommand{\calG}{{\mathcal G}}

\newcommand{\calO}{{\mathcal O}}

\newcommand{\calT}{{\mathcal T}}
\newcommand{\calU}{{\mathcal U}}

\newcommand{\calZ}{{\mathcal Z}}
\newcommand{\OO}{{\mathcal O}}

\newcommand{\frakm}{{\mathfrak m}}

\newcommand{\frakp}{{\mathfrak p}}
\newcommand{\frakq}{{\mathfrak q}}

\newcommand{\pp}{{\mathfrak p}}
\newcommand{\qq}{{\mathfrak q}}
\newcommand{\mm}{{\mathfrak m}}

\usepackage[mathscr]{euscript}

\newcommand{\scrX}{{\mathscr X}}

\DeclareMathOperator{\Char}{char}

\DeclareMathOperator{\Gal}{Gal}
\DeclareMathOperator{\Ind}{Ind}

\DeclareMathOperator{\Spec}{Spec}

\DeclareMathOperator{\Frac}{Frac}

\DeclareMathOperator{\rank}{rank}

\DeclareMathOperator{\Stab}{Stab}

\DeclareMathOperator{\ct}{ct}

\numberwithin{equation}{section}
\numberwithin{table}{section}

\newcommand{\defi}[1]{\textsf{#1}} 

\title{Number fields generated by points in linear systems on curves}
\author{Irmak Bal\c{c}{\i}k}
\address{Northwestern University, Department of Mathematics,  2033 Sheridan Road, Evanston, IL 60208, USA}
\email{irmak.balcik@northwestern.edu}

\author{Stephanie Chan}
\address{University College London, Department of Mathematics, Gower Street, London, WC1E 6BT, United Kingdom}
\email{stephanie.chan@ucl.ac.uk}

\author{Yuan Liu}
\address{University of Illinois at Urbana-Champaign, Department of Mathematics, 1409 W Green St, Urbana, IL 61820, USA}
\email{yyyliu@illinois.edu}

\author{Bianca Viray}
\address{University of Washington, Department of Mathematics, Box 354350, Seattle, WA 98195, USA}
\email{bviray@uw.edu}
\urladdr{http://math.washington.edu/\~{}bviray}

\date{\today}

\begin{document}

\begin{abstract}
    We develop techniques for determining the fibers of a morphism of curves \(\psi\colon C \to D\) over a nonarchimedean local field \(K\). These results have applications to studying closed point on curves over global fields since closed points on \(C\) of large degree or of very small degree are known to all arise as fibers of morphisms.
\end{abstract}

\maketitle
\section{Introduction}

Let \(C\) be a smooth projective geometrically integral curve over a number field \(k\). The Riemann--Roch theorem implies that if \(x\in C\) is a closed point of sufficiently large degree then \(x\) arises as the fiber of a morphism \(\psi\colon C \to \PP^1\) with \(\deg(x) = \deg(\psi)\). Conversely, given a degree \(d\) morphism \(\psi\colon C \to \PP^1\), Hilbert's irreducibility theorem implies that there are infinitely many \(t\in \PP^1(k)\) where the fiber \(C_t\) is integral, i.e., is a degree \(d\) point. In particular, \(C\) has infinitely many degree \(d\) points. At the same time, if the genus of \(C\) is at least \(2\), then Faltings's theorem implies that these degree \(d\) integral fibers must be defined over infinitely many degree \(d\) extensions of \(k\).

We ask: can we use the structure of \(\psi\) to understand these degree \(d\) extensions of \(k\) with a view to leveraging morphisms to understand all closed points on \(C\)? As a first step, we seek to understand the local behavior of these extensions, i.e., understanding the isomorphism classes of the \(k_v\) algebras \(\kk(C_t)\otimes_k k_v\). Since \(\PP^1(k)\) is dense in \(\PP^1(k_v)\), this question reduces to a question about covers of \(\PP^1\) over a local field.

In this paper, we develop techniques for determining the fibers of a morphism of  curves \(\psi\colon C \to D\) over a nonarchimedean local field \(K\). If \(\psi\) has \defi{good reduction} (Definition~\ref{def:GoodReductionOfCurves}) then for a closed point \(x\in C\) we show that the possible ramification degrees of the extension \(\kk(x)/\kk(\psi(x))\) of nonarchimedean local fields are determined by the combinatorial ramification data of the map \(\psi\). (The inertia degrees are known to be controlled by the fibers of the morphism over \(\F_v\) in accordance with Hensel's Lemma; see Proposition~\ref{prop:EtAlgUnramified} for more details.) Furthermore, we determine when every tamely ramified extension with allowed ramification degree \(e\) and inertia degree \(f\) does arise as an extension \(\kk(x)/\kk(\psi(x))\); this occurs exactly when the degree \(f\) unramified extension of \(K\) is disjoint from \(K(\mu_{e^{\infty}})\). These results have applications to studying closed point on curves over global fields, both in the case when \(D = \PP^1\) as explained above and also in the case when \(D\) has higher genus, as the so-called parameterized closed points on a curve \(C\) of very low degree arise as pullbacks from a unique morphism \(\pi\colon C \to D\)~\cite{VirayVogt}*{Theorem 6.0.1}.

To precisely state our results, we introduce some notation. Given a \(t\in D(K)\) outside the branch locus, the \(K\)-algebra \(\kk(C_t)\) factors as a product of fields, and this factorization is given by the decomposition of the \(K\)-scheme \(C_t\). 
Given a finite flat morphism \(\Psi\colon \calC \to \calD\) of \(\OO_K\)-schemes (where \(\OO_K\) denotes the valuation ring of \(K\)) whose generic fiber agrees with \(\psi\), then the maximal subfields of \(\kk(C_t)\) are partitioned as follows. Let \(\F_K\) denote the residue field of \(K\) and let \(\overline{t}\in \calD(\F_K)\) denote the specialization of \(t\); then 
\[
    \kk(C_t) = \prod_{\overline{x}\in \calC_{\overline{t}}}\kk(C_{t})_{\overline{x}}, \quad \textup{where}\quad
    \kk(C_{t})_{\overline{x}} := \prod_{\substack{x\in C_{t}\\ x\rightsquigarrow \overline{x}}}\kk(x).
\]

\begin{theorem}[Specialization of Theorem~\ref{thm:RamAndInertiaDegrees}]\label{thm:RamAndInertiaDegreesIntro}
    Let \(\psi\colon C \to D\) be a morphism of smooth projective geometrically integral curves over a nonarchimedean local field \(K\) that has good reduction (so there is an integral model \(\Psi\colon \calC \to \calD\); see Definition~\ref{def:GoodReductionOfCurves}), let \(t\in D(K)\) be outside the branch locus of \(\psi\), let \(\overline{t}\in \calD(\F_K)\) be the specialization of \(t\), and let \(\overline{x}\in \calC_{\overline{t}}\).
    Then every maximal subfield \(L\subset\kk(C_{t})_{\overline{x}}\) has the following numerical invariants.
    \begin{enumerate}
        \item The inertia degree \(f(L/K)\) is divisible by the degree of \(\overline{x}\).
        \item The ramification index \(e(L/K)\) divides the ramification index \(e(\overline{x}/\overline{t})\) of the morphism \(\Psi_{\F_K}\) at the point \(\overline{x}\).
        \item If \(e(L/K)=e(\overline{x}/\overline{t})\) (i.e., the ramification index \(e(L/k)\) is as large as possible), then \(L = \kk(C_{t})_{\overline{x}}\) and the inertia degree is \(f(L/K)=\deg(\overline{x})\).        
    \end{enumerate}
    Conversely, every extension \(L/K\) with ramification degree \(e:= e(\overline{x}/\overline{t})\) and inertia degree \(f:=\deg(\overline{x})\) is isomorphic to \(\kk(C_{t})_{\overline{x}}\) for some \(t\in \calD(K)\) specializing to \(\overline{t}\) \emph{if and only if} 
    \[
        \gcd\left(e, \frac{q^f - 1}{q-1}\right) = 1 \quad \mbox{where }q := \#\F_K.
    \]
\end{theorem}
In the case that \(\gcd\left(e, \frac{q^f - 1}{q-1}\right)\neq 1\), the more general Theorem~\ref{thm:RamAndInertiaDegrees} describes which extensions \(L/K\) with fixed numerical invariants are realized as \(\kk(C_{t})_{\overline{x}}\) for some point \(t\). The description therein shows that the proportion of realizable field extensions is  \(\frac{1}{\gcd\left(e, \frac{q^f - 1}{q-1}\right)}\).\footnote{\noindent This gcd measures the presence of ``new'' \(e^{\infty}\)-th roots of unity in the degree \(f\) unramified extension of \(K\).}
\begin{remark}
    If \(\psi\) is a Galois cover or if \(D = \PP^1\), then prior results of Beckmann, D\`ebes, Ghazi, K\"onig, Legrand, and Neftin essentially combine to give statements (1)--(3) from the theorem~\cites{Beckmann, DG2012, KLN-GrunwaldRamified, DL-TAMS, Legrand}. (Our notion of good reduction is not identical to the conditions that appear in these works but they differ up to a finite set of places.) If both assumptions hold (i.e., \(\psi\) is Galois and \(D= \PP^1\), then~\cite{KLN-GrunwaldRamified} gives a group-theoretic characterization of the extensions \(L/K\) that appear as \(\kk(C_{t})_{\overline{x}}\) for some \(t\in \PP^1(K)\). This group theoretic characterization is in terms of the relationship between the geometric inertia group of a ramification point embedded as a subgroup of \(\Gal(\psi)\) and the arithmetic inertia group \(\Gal(K^{tr}/K^{nr})\). See Section~\ref{sec:IntroComparison} for a comparison of the approaches.
\end{remark}

For morphisms with prescribed ramification behavior, we can apply Theorem~\ref{thm:RamAndInertiaDegrees} to each of the ramification points and assemble a more global description of the residue fields. For example, for general degree \(d\) morphisms (i.e., where the ramification indices are at most \(2\)), we have the following description.

\begin{thm}\label{thm:SdIntro}
    Let \(\psi \colon C \to \PP^1\) be a degree \(d\) morphism of smooth projective geometrically integral curves such that all points of \(C\) have ramification index at most \(2\) and let \(v\) be a nonarchimedean place of \(k\) with sufficiently large residue field.
    \begin{enumerate}
        \item \label{part:InertiaSd} For all partitions \((d_i)_{i=1}^r\vdash d\), there exists a \(t\in \PP^1(k)\) such that \(\kk(C_t)\) is a field with \(r\) places \(w_i|v\), each of inertia degree \(f(w_i/v) = d_i\).
        \item \label{part:RamificationUpperBoundSd} For every \(t\in \PP^1(k)\) outside the branch locus and every place \(w|v\) of \(\kk(C_t)\), the ramification degree \(e(w/v)\) is at most \(2.\) Furthermore, the number of ramified places is bounded above by 
        \(
             \#\{{\overline{x}} \in C_{{\overline{t}}} : e({\overline{x}/\overline{t}}) > 1\} 
        \) where \(\overline{t}\) is the specialization of \(t\).
        \item \label{part:RamificationExistenceSd} If \(\psi\) has a \(k_v\)-rational branch point and there is an odd degree ramification point \(x_v\in C\) lying above it, then for all quadratic ramified extensions \(L/\kk(x_v)\), there exists a \(t\in \PP^1(k)\) such that \(\kk(C_t)\otimes_k k_v\) contains \(L\) as a maximal subfield.
    \end{enumerate}
\end{thm}

For any cover \(C\to \PP^1\) whose combinatorial ramification data is fixed, we can obtain a result analogous to Theorem~\ref{thm:SdIntro}. Indeed, standard arguments involving spreading out, Hilbert irreducibility, and weak approximation (which we review in detail in Section~\ref{sec:Global}) reduce the proof of Theorem~\ref{thm:SdIntro} (and results analogous to it) to studying possible isomorphism classes of \'etale \(k_v\)-algebras \(\kk(C_t)\) for \(t\in \PP^1(k_v)\), which we may do using Theorem~\ref{thm:RamAndInertiaDegrees}.

\subsection{Method of proof and comparison to related work}\label{sec:IntroComparison}

We approach this problem as follows. First, we show that local ring extensions with a single source of ramification are always monogenic (Lemma~\ref{lem:MonogenicLocal}) and exploit the monogenicity to determine the isomorphism classes of residue fields in the local ring extension (Proposition~\ref{prop:RamifiedLocalPresentation}). (This proof relies on certain consequences of local class field theory, which are reviewed in Section~\ref{sec:LCFT}.) Our observation that the monogenicity of the local ring extensions together with standard tools of local class field theory and Newton polygons can be leveraged to 
give a comprehensive understanding of local ramified behavior is the key idea in the paper. The simplicity of the underlying ideas means that the method is very flexible and we expect that the method will be extendable to many other contexts, for example to study residue fields hyperplane sections of higher dimensional linear systems on curves. This method also does not require information about the Galois closure. This is especially valuable, since in our main applications of interest the morphisms will typically give rise to \(S_d\) covers, so computation of the Galois closure would be as computationally intensive as possible.

The remainder of the paper focuses on translating these ring theoretic results to our geometric situation of covers of curves. We show that the definition of a morphism of good reduction (Definition~\ref{def:GoodReductionOfCurves}) implies that the resulting extension of local rings has a single source of ramification (Lemma~\ref{lem:GoodReductionConseq}) and thus is a monogenic extension. Applying the ring-theoretic results gives a complete description of the isomorphism classes of \'etale \(k_v\)-algebras \(\kk(C_t)\) for \(t\) in an open neighborhood in \(\PP^1(k_v)\) (Theorem~\ref{thm:RamAndInertiaDegrees}).  
In Section~\ref{sec:Global}, we give the details of how these results are assembled to prove Theorem~\ref{thm:SdIntro}. Lastly, we give results that show how the points of \(\PP^1(k)\) whose fiber has specified local behavior are distributed in sets of bounded height (Section~\ref{sec:ModM}). These results are not necessary to prove Theorem~\ref{thm:RamAndInertiaDegrees} or Theorem~\ref{thm:SdIntro}, but they provide helpful context should the reader plan to use these Theorems~\ref{thm:RamAndInertiaDegrees} and~\ref{thm:SdIntro} in an explicit way.

As indicated above, the main novelty in the paper is in observing that the local ring extensions given by a cover of \emph{arbitrary} curves are monogenic and that this can be exploited to determine the isomorphism classes of the ramified residue field extensions. In other words, the main novelty is in the last part of Theorem~\ref{thm:RamAndInertiaDegreesIntro} (specialization of Theorem~\ref{thm:RamAndInertiaDegrees}\eqref{part:CurveRamIsom}). Indeed, the characterizations of the unramified residue field extensions are simple consequences of the definition of \'etale extensions of \(p\)-adic valuation rings and the Chebotarev density theorem. We include these details in Proposition~\ref{prop:EtAlgUnramified} solely for the reader's convenience. Additionally, the proof of Theorem~\ref{thm:SdIntro} from Theorem~\ref{thm:RamAndInertiaDegrees} relies on the standard techniques of spreading out, Hilbert's irreducibility theorem, weak approximation, and sometimes Chebotarev's density theorem (depending on the situation). Again, we have included the details solely for the convenience of the reader, to make the paper as self-contained as possible.

To the best of our knowledge, there is no direct link from our approach to the techniques used by Beckmann, D\`ebes, Ghazi, K\"onig, Legrand, and Neftin~\cites{Beckmann, DG2012, KLN-GrunwaldRamified, DL-TAMS, Legrand} in their various papers. Their work was motivated by the problems about \(G\)-extensions, particularly the Grunwald problem and the parameterization problem. As such, they were primarily interested in \(G\)-covers of \(\PP^1\) and in understanding the group theoretic criteria that controlled the isomorphism classes of specializations. In contrast, our motivation stems from the arithmetic of closed points on curves. Closed points of large degree (by Riemann-Roch) or of very small degree (by~\cite{VirayVogt}*{Theorem 6.0.1}) arise as complete fibers of morphisms \(\pi\colon C \to D\), but these morphisms are typically not Galois and the base need not be \(\PP^1\). As such, we are most interested in criteria based on the geometry of \(\pi\) and endeavored to use as little information about the Galois closure as possible.

\subsection*{Notation}

    Throughout the paper, we use \(K\) to denote a nonarchimedean local field, \(\OO_K\) its valuation ring, and \(\F_K\) its residue field.  We let \(\pi_K\) denote a uniformizer and write \(q_K := \#\F_K\). When the local field \(K\) is clear from context, we will omit the subscript on \(\OO_K, \F_K, \pi_K,\) and \(q_K\).  We write \(K^{tr}\) for the maximal tamely ramified extension of \(K\) and \(K^{nr}\subset K^{tr}\) for the maximal unramified subextension. Given any positive integer \(d\) and a local field \(K\), we write \(K(d)\) for the unique unramified degree \(d\) extension of \(K\).  Similarly, if \(\F\) is a finite field, we write \(\F({d})\) for the unique degree \({d}\) extension of \(\F\). We will often use that any local \'etale extension of \(\OO_K\) is isomorphic to \(\OO_{K({d})}\) for some \({d}\) and this in turn implies that any (not necessarily local) \'etale extension of \(\OO_K\) is a product \(\prod_{d_i}\OO_{K(d_i)}\) for some integers \(d_i\).

    Since $\calO_K$ is locally compact, it comes equipped with a Haar measure, which we denote by \(\mu_{\OO_K}\). Since \(\PP^1\) can be obtained by gluing together two copies of \(\A^1\) and \(\A^1(\OO_K) \simeq \OO_K\), we can extend \(\mu_{\OO_K}\) to a Haar measure \(\mu_V\) for \(V\) a reduced subscheme of \(\PP^1_{\OO_K}\).

    Given a proper \(\OO_K\) scheme \(\scrX \to \Spec \OO_K\), and a closed point on the generic fiber \(x\in \scrX_K\), the valuative criterion of properness allows us to extend to a map \(\Spec \OO_{\kk(x)}\to \scrX\). In particular, restricting this morphism to the unique closed point of \(\Spec \OO_{\kk(x)}\) gives a closed point \(\overline{x}\in\scrX\). We say this point \(\overline{x}\) is the \defi{specialization} of \(x\) and write \(x\rightsquigarrow \overline{x}\) to indicate that \(x\) specializes to \(\overline{x}\).
    
    We reserve \(k\) for number fields, and denote the set of places of \(k\) by \(\Omega_k\). Given a nonarchimedean \(v\in \Omega_k\), we write \(k_v\) for the completion and \(\OO_v, \F_v, \pi_v\) for its valuation ring, residue field, and a choice of uniformizer, respectively.
    
    Given any finite extension of fields \(L/F\), we write \(N_{L/F}\colon L^{\times} \to F^{\times}\) for the norm map. Given a finite extension of local fields \(L/K\), we write \(e(L/K)\) and \(f(L/K)\) for the ramification and inertia degrees, respectively. In other words,
    \[
        f(L/K) = [\F_L : \F_K]\quad \textup{and}\quad \pi_L^{e(L/K)}\OO_L = \pi_K\OO_L.
    \]

    We also consider the ramification degree \(e(\mm_S/\mm_R)\) of a finite extension of local rings \((R, \mm_R) \hookrightarrow (S, \mm_S)\), which we define to be the ratio \(\dim_{R/\mm_R}(S/\mm_RS)/\dim_{R/\mm_R}(S/\mm_S)\). If \(R\) and \(S\) are discrete valuation rings and their extension of fraction fields is separable, then this agrees with the usual ramification degree which measures the index of value groups.

    For a profinite group \(G\) and an element \(g\in G\), we write \(\langle g \rangle\) for the closed subgroup topologically generated by \(g\). For a finite subgroup \(G\) of \(S_d\), we may write every element \(g\in G\) as a product of disjoint cycles, so we define the \defi{cycle type} of \(g\), denoted \(\ct(g)\), to be the multiset of lengths of the disjoint cycles. The cycle type always gives a partition of \(d\), which we denote \(\ct(g)\vdash d.\)

    Given a smooth proper scheme of finite type \(X\) and a codimension \(1\) subscheme \(Z\subset X\), we write \(D_Z := \sum_{D\subset Z} n_D(Z) D\)
    be the associated effective divisor on \(X\), where \(D\) runs through the irreducible components of \(Z\), and we write \(\kk(Z)\) for the 
    ring
    \[
    \kk(Z) = \kk(D_Z) := \prod_{D} \OO_{X,D}/\mm_D^{n_D(Z)}.
    \]
    Note that if \(Z\) is integral, then \(\kk(Z)\) is the residue field of \(Z\) and if \(Z\) is reduced, then \(\kk(Z)\) is the ring of global sections of the sheaf of total quotient rings.

    Recall that for any field \(F\), there is a bijection between field extensions \(L/F\) of transcendence degree \(1\) and normal projective curves \(C/F\)~\cite{QingLiu}*{Proposition 7.3.13}. Thus, we will abuse notation and use Galois properties of covers of curves to mean the corresponding function field properties. For example, given a separable morphism \(\psi\colon C \to \PP^1\) of normal projective curves, we say that the \(\psi\) is Galois if the corresponding extension of function fields is Galois. Further, the Galois closure of \(\psi\) is a morphism of normal projective curves \(\widehat \psi \colon \widehat C \to \PP^1\) such that the extension of function fields from \(\widehat\psi\) is the Galois closure of the map of function fields given by \(\psi\). We also write \(\Gal(\widehat{C}/\PP^1)\) and \(\Gal(\widehat{C}/C)\) to mean the Galois groups of the function field extensions.

\section*{Acknowledgements}
    This project was started as part of the Women in Numbers 6 (WIN6) conference held at Banff International Research Station. We thank the organizers of WIN6, Shabnam Akhtari, Alina Bucur, Jennifer Park, and Renate Scheidler, and the staff and administration of BIRS for this opportunity.  We thank Andrew Obus for a helpful discussion regarding Definition~\ref{def:GoodReductionOfCurves} and thank Samir Siksek for pointing out the work of Beckmann, which can be used to slightly strengthen Proposition~\ref{prop:RamifiedLocalPresentation} and Theorem~\ref{thm:RamAndInertiaDegrees} (see Remarks~\ref{rmks:RamifiedLocalPresentation}).
    
    Part of the work in this project was supported by National Science Foundation grant DMS-1928930 while the last three authors were in residence at the Simons Laufer Mathematical Sciences Institute in Berkeley, California. YL was partially supported by NSF grant DMS-2200541. BV was supported in part by an AMS Birman Fellowship and NSF grant DMS-2101434, and thanks the University of Washington ADVANCE office for their Write Right Now program during which some of this paper was written.

\section{Structural results on tamely ramified extensions of local fields}\label{sec:LCFT}

The goal of this section is to give a characterization of isomorphism classes of totally tamely ramified extensions in terms of uniformizers in the norm subgroups. This characterization relies on results from local class field theory and Iwasawa's description of the Galois group of the maximal tamely ramified extension.
\begin{prop}\label{prop:IsoTotTamelyRam} 
    Let $K$ be a nonarchimedean local field, let $L_1,L_2$ be totally tamely ramified extensions of $K$, and for \(i=1,2\), let \(K\subset F_i \subset L_i\) be the maximal subextensions that are abelian over \(K\).
    
    If \([L_1:K] = [L_2:K]\), then the following are equivalent:
    \begin{enumerate}
        \item \label{it:Conj} $L_1 = \sigma(L_2)$ for some \(\sigma\in \Gal(\Kbar/K)\). 
        \item \label{it:AbelianEqual} \(F_1 = F_2\). 
        \item\label{it:NormAbelian} \(N_{F_1/K}(F_1^{\times}) = N_{F_2/K}(F_2^{\times})\).
        \item \label{it:NormTTRam} $N_{L_1/K}(L_1^{\times}) = N_{L_2/K}(L_2^{\times})$.
        \item \label{it:UniformizerAbelian}  For any choice of uniformizers \(\pi_i\in N_{F_i/K}(F_i^{\times})\), the ratio \(\pi_1/\pi_2\in \OO_{K}^{\times [F_1:K]}\).
        \item \label{it:UniformizerTTRamModified}  For any choice of uniformizers \(\pi_i\in N_{L_i/K}(L_i^{\times})\), the ratio \(\pi_1/\pi_2\in \OO_{K}^{\times [F_1:K]}\). 
        \item \label{it:UniformizerTTRam}  For any choice of uniformizers \(\pi_i\in N_{L_i/K}(L_i^{\times})\), the ratio \(\pi_1/\pi_2\in \OO_{K}^{\times [L_1:K]}\). 
    \end{enumerate}
\end{prop}
\begin{proof} We are mainly interested in the equivalence of~\eqref{it:Conj} and~\eqref{it:UniformizerTTRam}. However, the tools which allow us to relate these two statements naturally give us the equivalent statements~\eqref{it:AbelianEqual}--\eqref{it:UniformizerTTRamModified}, as we now explain.
The equivalence of~\eqref{it:AbelianEqual} and~\eqref{it:NormAbelian} is a special case of the Main Theorem of Local Class Field Theory (abelian extensions are characterized by their norm groups); see, e.g., \cite{Neukirch}*{Chapter~V, Theorem~1.4}. Furthermore, using the local reciprocity map, we can deduce $N_{F_i/K}(F_i^{\times})\cap \calO_K^{\times}=\calO_K^{\times [F_i:K]}$ which yields the equivalence of~\eqref{it:NormAbelian} and~\eqref{it:UniformizerAbelian}. The Norm Limitation Theorem~\cite{Grant-Leitzel}*{Theorem A} implies that \(N_{F_i/K}(F_i^{\times}) = N_{L_i/K}(L_i^{\times})\) which yields the equivalences~\eqref{it:NormAbelian}\(\Leftrightarrow\)\eqref{it:NormTTRam} and~\eqref{it:UniformizerAbelian}\(\Leftrightarrow\)\eqref{it:UniformizerTTRamModified}. Therefore, we have ascertained all equivalences among~\eqref{it:AbelianEqual}--\eqref{it:UniformizerTTRamModified}. Additionally, the implication~\eqref{it:Conj}\(\Rightarrow\)\eqref{it:AbelianEqual} follows from the Galois theory correspondence. It remains to prove ~\eqref{it:AbelianEqual}\(\Rightarrow\)\eqref{it:Conj} and~\eqref{it:UniformizerTTRamModified}\(\Leftrightarrow\)\eqref{it:UniformizerTTRam}.

The implication~\eqref{it:AbelianEqual}\(\Rightarrow\)\eqref{it:Conj} is proved in Corollary~\ref{cor:GaloisTotallyTamelyRamified}\eqref{part:conj}, which uses Iwasawa's presentation of \(\Gal(K^{tr}/K)\). Finally, Iwasawa's presentation of the Galois group also implies that \([F_i:K] = \gcd([L_i:K], q_K-1)\) (Corollary~\ref{cor:GaloisTotallyTamelyRamified}\eqref{part:index}) so \(\OO_K^{\times [L_1:K]} = \OO_K^{\times \gcd([L_1:K], q_K-1)} = \OO_K^{\times [F_1:K]}\) and this gives the equivalence of \eqref{it:UniformizerTTRamModified} and~\eqref{it:UniformizerTTRam}.
\end{proof}

\subsection{Consequences of Iwasawa's presentation of \(\Gal(K^{tr}/K)\)}\label{sec:Iwasawa}
In 1955, Iwasawa gave the following metacyclic presentation of the Galois group of the maximal tamely ramified extension of \(K\).
\begin{thm}[\cite{Iwasawa55}]\label{thm:Iwasawa} Let \(K\) be a nonarchimedean local field.
The Galois group \(\Gal(K^{tr}/K^{nr})\) is isomorphic to {\(\prod_{p\nmid q_K} \Z_p\)} and the quotient \(\Gal(K^{nr}/K)\simeq \Gal(K^{tr}/K)/\Gal(K^{tr}/K^{nr})\) is isomorphic to \(\widehat\Z\). Moreover, there is a generator \(\tau\) of \(\Gal(K^{tr}/K^{nr})\) and an element \(\sigma\in \Gal(K^{tr}/K)\) whose image generates the quotient \(\Gal(K^{nr}/K)\) such that 
    \[
        \Gal(K^{tr}/K) = \langle \tau, \sigma \,|\, \sigma\tau\sigma^{-1} = \tau^{q_K}\rangle.
    \]
    \end{thm}
From this theorem we can deduce information about Galois groups of subextensions.
\begin{cor}\label{cor:Iwasawa}
    Retain the notation from Theorem~\ref{thm:Iwasawa}.
    If \(L/K\) is a tamely ramified extension and \(H:= \Gal(K^{tr}/L)\), we have
    \[
        H\cap \langle \tau\rangle = \langle \tau^{e(L/K)}\rangle
        \quad \textup{and} \quad
        \left|\Gal(K^{tr}/K)/\langle\tau\rangle H\right| = f(L/K).
    \]
    If \(L/K\) is totally tamely ramified, then \(H = \langle \tau^{e(L/K)}, \tau^i\sigma\rangle\) for some \(0 \leq i < e(L/K)\).
\end{cor}
\begin{proof}
    By definition, \(e(L/K) = [L : L\cap K^{nr}]\) and \(f(L/K) = [L\cap K^{nr}:K]\). Thus, the first statement follows from Theorem~\ref{thm:Iwasawa} and the Galois correspondence. Further, if \(L/K\) is totally tamely ramified, then \(\Gal(K^{tr}/K)=\langle\tau\rangle H\) so \(H\) must contain \(\tau^i\sigma\) for some \(i\). 
\end{proof}

Now that we have presentations for the subgroups fixing totally tamely ramified extensions, determining the Galois action on totally tamely ramified extensions of \(K\) is a purely group-theoretic exercise.

\begin{lemma}\label{lem:ProfiniteMetacyclic}
    Let \(q\) be a prime power and let \(\calG\) be the profinite group generated by two elements \(\sigma, \tau\) with the only relation   
    \begin{equation}\label{eq:Iwasawa}
    \sigma \tau \sigma^{-1} = \tau^q.
    \end{equation}
    Let \(e\) be a positive integer that is coprime to \(q\).
    For integers \(0\leq i < e\), define \(H_i := \langle \tau^e, \tau^i\sigma\rangle\).  Then \(H_i\) is conjugate to \(H_j\) if and only if \(i \equiv j \pmod{\gcd(e, q-1)}\).
\end{lemma}
\begin{remark}
    Note that by~\eqref{eq:Iwasawa}, \(\langle \tau^i\rangle\) is a normal subgroup of \(\calG\) for all \(i\), and thus that \(\langle \tau^{qi}\rangle = \langle \tau^i\rangle\) for all \(i\). 
\end{remark}
\begin{proof}
    For a fixed \(i\), we compute that \(\tau^{-1}\tau^{i}\sigma\tau = \tau^{i + q-1}\sigma\). Thus, \(H_i\) is conjugate to \(H_{i + (q-1)\bmod e}\). By iterating, we deduce that \(H_i\) is conjugate to \(H_j\) if \(i \equiv j \pmod{\gcd(e, q-1)}\). Since \(\calG\) is generated by \(\sigma\) and \(\tau\), it suffices to show that if \(\sigma H_i\sigma^{-1} = H_j\) then \(i\equiv j \pmod{\gcd(e, q-1)}\). We compute 
    \[
    \sigma H_i\sigma^{-1} = \langle\sigma\tau^e\sigma^{-1}, \sigma\tau^i\sigma \sigma^{-1}\rangle = \langle \tau^{qe}, \tau^{qi}\sigma\rangle = \langle \tau^{e}, \tau^{qi}\sigma\rangle = H_{qi\bmod e}.
    \]
    This completes the proof, since \(qi = i + (q-1)i\equiv i \pmod{q-1}\).
\end{proof}

\begin{cor}\label{cor:GaloisTotallyTamelyRamified}
    Let \(K\) be a nonarchimedean local field, let \(L/K\) be a totally tamely ramified degree \(e\) extension, let \(K\subset F\subset L\) be the maximal Galois subextension contained in \(L\). Then:
    \begin{enumerate}
        \item \(\Gal(F/K)\) is cyclic,
        \item \( [F:K] = \gcd\left(e, q_K - 1\right),\) and\label{part:index}
        \item if \(L'/K\) is a totally tamely ramified degree \(e\) extension that contains \(F\), then \(L' = g(L)\) for some \(g\in \Gal(K^{tr}/K)\). \label{part:conj}
        \end{enumerate}
\end{cor}
\begin{proof}
    The extension \(F/K\) is totally ramified and Galois, thus \(\Gal(F/K)\) is a quotient of the cyclic profinite group \(\Gal(K^{tr}/K^{nr})\) and  hence \(\Gal(F/K)\) is cyclic.
    
    Let \(H := \Gal(K^{tr}/L)\). Then by Corollary~\ref{cor:Iwasawa}, \(H\) is generated by \(\tau^e\) and \(\tau^i\sigma\) for some \(0\leq i < e=[L:K]\). Thus, Lemma~\ref{lem:ProfiniteMetacyclic} implies that \(H\) is a normal subgroup if and only if \([L:K]\mid q_K - 1\). Applying the same argument to subextensions of \(L\) we deduce that \([F:K] = \gcd(e, q_K - 1)\). 

    To prove the last statement, it suffices to show there exists a \(g\in \Gal(K^{tr}/F)\) such that \(L' = g(L)\), thus we may replace \(K\) with \(F\) and \(e\) with \(e/\gcd(e, q_K - 1)\) which equals \([L:F]\). In particular, we may assume that \(L/K\) has no nontrivial Galois subextensions. By the previous statement, this implies that \(\gcd(e, q_K - 1) =1\), which by Lemma~\ref{lem:ProfiniteMetacyclic} implies that \(\langle\tau^e, \tau^i\sigma\rangle\) is conjugate to \(\langle \tau^e, \tau^j\sigma \rangle\) for all \(i,j\). By Corollary~\ref{cor:Iwasawa}, \(\Gal(K^{tr}/L)\) and \(\Gal(K^{tr}/L')\) both have this form, so the Galois correspondence yields the desired isomorphism.
\end{proof}

\section{Extensions of regular local rings}\label{sec:LocalRingExtns}

    \subsection{Sufficient criteria for monogenicity}

Finite extensions of discrete valuation rings with separable residue fields are always monogenic.  Monogenicity can fail, however, for higher dimensional regular local rings.  For example, \(F[[\sqrt{x}, \sqrt{y}]]\) cannot be generated by a single element over \(F[[x,y]]\).\footnote{
By Nakayama's lemma, it suffices to show that there is no power basis for \(F[[\sqrt{x}, \sqrt{y}]]/\langle x, y\rangle\), which we can do by showing that \(1,\beta, \beta^2, \beta^3\) are linearly dependent in \(F[[\sqrt{x}, \sqrt{y}]]/\langle x, y\rangle\) for any \(\beta = c_0 + c_1\sqrt{x} + c_2\sqrt{y} + c_3\sqrt{xy}\).} Informally, the issue in this example is that there is ramification along two different axes. However, if all ramification is captured by the ramification at a single height \(1\) prime, then regular local ring extensions are monogenic.

    \begin{lemma}\label{lem:MonogenicLocal}
        Let \((R,\frakm_R)\) be a regular local ring and let \((S,\frakm_S)\) be a regular local finite extension of \(R\). If there exists a height \(1\) prime \(\frakq_0\subset S\) such that \(R/\frakq_0\cap R \to S/\frakq_0\) is \'etale,
        then \(S\) is monogenic over \(R\), i.e., there exists a \(\beta\in S\) such that \(R[\beta] = S\).  
    \end{lemma}

    Under additional assumptions which will hold in our main application (i.e., if \(R\) is an extension of \(\OO_K\), the prime ideal \(\frakq_0\subset S\) lies over the zero ideal of \(\OO_K\), and \(\frakq_0\) is the unique prime ideal lying over \(\frakq_0\cap R\)), then we can say more about the minimal polynomial of \(\beta\) (see Lemma~\ref{lem:RegularLocalRingPresentation} below).

    The proof of Lemma~\ref{lem:MonogenicLocal} that we give below relies on the well-known fact that finite \'etale extensions of local rings are monogenic. We are not aware of a reference that provides the exact form we require, so we provide a proof for the reader's convenience.
    \begin{lemma}\label{lem:FiniteEtaleIsLocallyMonogenic}
        Let \((R,\frakm_R)\) be a regular local ring and let \((S,\frakm_S)\) be a regular local finite extension of \(R\). If \(R\to S\) is \'etale, then there exists a \(\beta\in S\) such that \(R[\beta]=S\). Furthermore, if \(f(z)\in R[z]\) is the minimal polynomial of \(\beta\), then \(f'(\beta)\in S^{\times}\).
    \end{lemma}
    \begin{proof}
        Since \(R\to S\) is \'etale and \(S\) is local, \(\frakm_R S = \frakm_S\). Furthermore, \'etale maps are preserved under base change, so \(R/\mm_R \to S\otimes R/\mm_R = S/\mm_S\) is \'etale. Thus, \(S/\mm_S\) is a separable extension of \(R/\mm_R\) 
        so by the primitive element theorem 
        there exists a \(\beta_0\in S/\mm_S\) such that \(S/\mm_S = R/\mm_R[\beta_0]\) and whose minimal polynomial \(f_0(z)\in R/\mm_R[z]\) is separable. In particular, \(f'_0(\beta_0)\not\equiv 0 \bmod \mm_S\).  If \(\beta\in S\) is a lift of \(\beta_0\), then we have: 
        \[
        R[\beta]/\mm_R \simeq (R/\mm_R)[\beta\bmod \mm_R S] = (R/\mm_R)[\beta\bmod \mm_S] = (R/\mm_R)[\beta_0] = S/\mm_S = S/\mm_RS.
        \]
        Thus, by Nakayama's lemma, \(R[\beta] = S\).  If \(f(z)\in R[z]\) denotes the minimal polynomial of \(\beta\), then \(f'(\beta)\bmod \mm_S= f_0'(\beta_0)\neq 0\). Therefore \(f'(\beta)\notin \mm_S\) and hence is a unit.
    \end{proof}

    \begin{proof}[Proof of Lemma~\ref{lem:MonogenicLocal}]
        If \(R\to S\) is \'etale, then we apply Lemma~\ref{lem:FiniteEtaleIsLocallyMonogenic}. Henceforth we assume that \(R\to S\) is not \'etale and let \(\frakp_0 := \frakq_0 \cap R \). 
        Denote $R_0:=R/\frakp_0$, $\mm_{R_0}:=\mm_R/\frakp_0$, $S_0:=S/\frakp_0$, and $\mm_{S_0}:=\mm_S/\frakp_0$.
        By assumption, \(R_0  \to S_0\) is \'etale, so Lemma~\ref{lem:FiniteEtaleIsLocallyMonogenic} implies that there is \(\beta_0\in S_0\) such that \(R_0[\beta_0] = S_0\) and furthermore, that  \(f_0'(\beta_0) \in S_0^{\times}\) where \(f_0(x)\) denotes the minimal polynomial of \(\beta_0\). Moreover, since \(R_0\to S_0\) is \'etale, \(\mm_{S_0} = \mm_S/\frakq_0 = \mm_{R_0} S_0 = (\mm_R/\frakp_0)(S/\frakq_0)\), i.e., \(\mm_S = \frakq_0 + \mm_RS \).

        Choose \(\beta\in S\) to be a lift of \(\beta_0\) and choose \(f_1(z)\in R[z]\) to be a monic lift of \(f_0\). Note that this implies that \(f_1\bmod \pp_0\) is the minimal polynomial of \(\beta\bmod \pp_0\), but \(f_1\) may not be the minimal polynomial of \(\beta\). By Nakayama's Lemma, \(R[\beta]=S\) if and only if \(R/\frakm_R[\beta \bmod \mm_R] = S/\frakm_R S\). Consider the following map of exact sequences:
        \[
        \begin{tikzcd}
             0 \arrow[r] & {\dfrac{ R[\beta]\cap \mm_S}{\mm_R[\beta]}} 
             \arrow[r]\arrow[d] & 
             {\dfrac{R}{\frakm_R}[\beta \bmod \mm_R]} 
             \arrow[r]\arrow[d, hook] & \dfrac{R}{\frakm_R}[\beta \bmod \mm_S] 
             \arrow[r]\arrow[d] 
            & 0\\
            0 \arrow[r] & \dfrac{\mm_S}{\mm_RS} \arrow[r] & \dfrac{S}{\mm_RS} \arrow[r] & \dfrac{S}{\mm_S} \arrow[r] & 0.
        \end{tikzcd}
    \]
        Since \(\beta\bmod \mm_R = \beta_0\bmod \mm_{R_0}\), the rightmost vertical map is a bijection. Additionally, since \(f_1(x)\bmod \mm_R\) 
        is the minimal polynomial of \(\beta_0\bmod\mm_R\), \(f_1(\beta) \in  R[\beta]\cap\frakq_0\subset R[\beta]\cap \mm_S\). Thus, if \(f_1(\beta)\bmod\mm_R\) generates \(\mm_S/\mm_R\)
        then the leftmost vertical arrow is a bijection and hence we have the desired equality.

        Assume that \(f_1(\beta)\bmod\mm_R\) does not generate \(\mm_S/\mm_R S = (\frakq_0 + \mm_RS)/\mm_R S\simeq  \frakq_0/\mm_R S\cap \frakq_0 = \frakq_0/\frakp_0S\). 
        By assumption, \(S\) is a regular local ring hence a unique factorization domain, so there exists a \(q_0\in S\) that generates the height \(1\) ideal \(\frakq_0\). 
        Since
        \(f_1(\beta)\) does not generate \(\frakq_0\), 
        \(f_1(\beta) = q_0a\) for some \(a\in {\frakm_S}\) and so
        \[
        f_1(\beta + q_0)= f_1(\beta) + f_1'(\beta)q_0 + q_0^2b = q_0(a + f_1'(\beta) + q_0b),
        \]
        for some \(b\in S\). Further, since \(f_1'(\beta)\bmod \frakq_0 = f_0'(\beta_0)\in S_0^{\times}\), \(f_1'(\beta)\notin \mm_S\); thus \(\langle f_1(\beta + q_0)\rangle = \frakq_0\).
        Therefore, we may replace \(\beta\) with \(\beta + q_0\) and conclude that \(R[\beta] = S\).
    \end{proof}

    \begin{lemma}\label{lem:RegularLocalRingPresentation}
        Retain the notation and assumptions from Lemma~\ref{lem:MonogenicLocal},  let \(f(z)\in R[z]\) be the minimal polynomial of \(\beta\), and let \(\frakp_0 := \frakq_0\cap R\). Assume, in addition, that 
        \begin{enumerate}
            \item the map \(R\hookrightarrow S\) is not \'etale,
            \item there is an inclusion \(\OO_K\hookrightarrow R\) whose composition with the quotient \(R\to R/\frakp_0\) yields an isomorphism \(\OO_K\xrightarrow{\sim}R/\frakp_0\), and
            \item the prime \(\frakq_0\) is the unique prime lying over \(\frakp_0\).
        \end{enumerate} 
        Then, there exists a monic
        polynomial \(\tilde{f}\in \OO_K[z]\) that is irreducible modulo \(\pi\) and an integer \(e>1\) such that 
        \[
            f - \tilde{f}^e = p_0 s\in R[z], 
        \]
        where \(p_0\in R\) is a generator of \(\pp_0\) and \(s\in R[z]\) is such that \(s \bmod f\) is a unit in \(S\).
    \end{lemma}
    \begin{proof}
        By Lemma~\ref{lem:MonogenicLocal}, \(S\simeq R[z]/f(z) \) so \(S/\frakp_0S \simeq (R/\frakp_0)[z]/f(z)\bmod\pp_0\). Since \(\frakq_0\) is the unique prime lying over \(\frakp_0\), the image of the monic polynomial \(f\bmod \frakp_0\) in \(\OO_K\) is a prime power.  Further, since \(R\hookrightarrow S\) is not \'etale at \(\frakq_0\), it is a nontrivial prime power, i.e., \(f\bmod \frakp_0 = \tilde{f}(z)^e\) for some \(e>1\) and irreducible polynomial \(\tilde{f}\in \OO_K[z]\).  Thus, \(\frakq_0 = \frakp_0 S + \langle \tilde{f}\rangle\). Since \(\OO_K\simeq R/\frakp_0 \to S/\frakq_0 \simeq \OO_K[z]/\tilde{f}(z)\) is \'etale and \(\OO_K\) is Henselian, the polynomial \(\tilde{f}\) must remain irreducible modulo \(\mm_R\). In particular, \(\frakm_S = \frakm_R S + \langle \tilde f\rangle \subset R[z]/f\).

        The assumption that \(S=R[z]/f\) is a regular local ring implies that \(f\notin (\mm_R[z] + \langle \tilde{f}\rangle)^2\) or equivalently (since \(e>1\)) that \(f - \tilde{f}^e\notin (\mm_R[z] + \langle \tilde{f}\rangle)^2\). Since \(f - \tilde{f}^e \in \frakp_0[z]\subset \mm_R[z]\), for any generator \(p_0\) of \(\frakp_0\), we may write \( f - \tilde{f}^e = p_0 \cdot s \) for some \(s\in R[z]\) that is not in \(\mm_R[z] + \langle\tilde{f}\rangle\). This implies that \(s \bmod f\in S\) is a unit, as desired.
    \end{proof}
    
    \subsection{Ramified monogenic extensions of \(2\)-dimensional regular local rings}\label{sec:RamifiedLocal}

    In this section, we restrict to the case of finite extensions of \(2\)-dimensional regular local rings whose ramification is concentrated at a single prime ideal and that are \(1\)-dimensional extensions of \(\OO_K\). We use Lemmas~\ref{lem:MonogenicLocal} and~\ref{lem:RegularLocalRingPresentation} to determine the isomorphism classes of \(S_{\frakp}/\frakp S_{\frakp}\) for height \(1\) primes \(\frakp\subset R\).

\begin{prop}\label{prop:RamifiedLocalPresentation}
    Let \(R\hookrightarrow S\) be a finite extension of \(2\)-dimensional regular local rings where \(R\supset \OO_K\) and this extension is relative dimension \(1\). Assume that \(R\to S\) satisfies the assumptions of Lemmas~\ref{lem:MonogenicLocal} and~\ref{lem:RegularLocalRingPresentation}, which implies the existence of primes \(\qq_0\subset S\) and \(\pp_0:=\qq_0\cap R \subset R\) such that \(R/\frakp_0\to S/\frakq_0\) is \'etale and \(\frakq_0\) is the unique prime lying over \(\frakp_0\).

    Then for any generator \(p_0\in R\) of \(\frakp_0\) there is a bijection:
    \begin{equation}\label{eq:OK-primes}
     \OO_K\longleftrightarrow \left\{\frakp\subset R :\textup{ht}(\frakp) = 1, \OO_K\hookrightarrow R \to R/\frakp \text{ is an isomorphism} \right\}, \; a\mapsto \frakp_a:= \langle p_0 + \pi a\rangle.
    \end{equation}
    Further, for \(a\in \OO_K \smallsetminus\{0\}\), the \'etale algebra 
    \(\kk(S_{\frakp_a}) := S_{\frakp_a}/\frakp_a S_{\frakp_a}\) has the following properties. 
    \begin{enumerate}
        \item \label{part:Dimension}
        \(\dim_{K} \kk(S_{\frakp_a})  = e(\mm_S/\mm_R)\cdot [S/\mm_S : R/\mm_R]\). 
        
        \item \label{part:Inertia} 
        Any maximal subfield \(L\subset \kk(S_{\frakp_a})\) has inertia degree \(f(L/K)\geq [S/\mm_S:R/\mm_R]\).
        
        \item \label{part:Ramification}
        Any maximal subfield \(L\subset \kk(S_{\frakp_a})\) has ramification degree $e(L/K)$ such that
        \[
        \frac{e(\mm_S/\mm_R)}{\gcd(1 + v(a), e(\mm_S/\mm_R))}\leq e(L/K)\leq e(\mm_S/\mm_R).
        \]
        
        \item \label{part:DistancefrakptoB} 
        If \(1 + v(a)\) is relatively prime to \(e(\mm_S/\mm_R)\) (or more generally if \(e(\mm_S/\mm_R) = e(L/K)\)), then \(\kk(S_{\frakp_a})\) is a field with ramification degree equal to \(e(\mm_S/\mm_R)\) and inertia degree equal to \([S/\mm_S:R/\mm_R]\).
        
        \item \label{part:RamIsom} 
         Assume that \(e: = e(\mm_S/\mm_R)\) is coprime to \(q_K\), which by~\eqref{part:Ramification} implies that all maximal subfields \(L\subset \kk(S_{\frakp_a})\) are tamely ramified over \(K\). Let \(K' := K([S/\mm_S:R/\mm_R])\) and fix a uniformizer \(\pi\) of \(K\). Then there exists a subgroup \(\OO_{K}^{\times}\OO_{K'}^{\times e}/\OO_{K'}^{\times e}<H< \OO_{K'}^{\times}/\OO_{K'}^{\times e}\) (necessarily cyclic) that completely characterizes the isomorphism classes of field extensions \(\kk(S_{\frakp_a})/K\) for \(\frakp_a\) with \(1+v(a)\) coprime to \(e(\mm_S/\mm_R)\). More precisely:
        \begin{enumerate}
            \item \label{part:Existence}
            If \(u\in \OO_{K'}^\times\) generates \(H/\OO_K^{\times}\OO_{K'}^{\times e}\) and \(L/K'\) is a totally tamely ramified extension of degree \(e(\mm_S/\mm_R)\) with \(u\pi\in N_{L/K}(L^{\times})\), then there exists an \(a\in \OO_K\) such that \(L\simeq \kk(S_{\frakp_{a}})\).
            
            \item \label{part:Uniqueness}
            Conversely, if \(1+v(a)\) is coprime to \(e(\mm_S/\mm_R)\) and \(\varpi_{\frakp_a}\in N_{\kk(S_{\frakp_a})/K'}(\kk(S_{\frakp_a})^{\times})\) is a uniformizer, then \(\varpi_{\frakp_a}\pi^{-1}\) generates \(H/\OO_K^{\times}\OO_{K'}^{\times e}\).

            \item \label{part:PartitionByv} 
            Furthermore, if \(a,a'\in\OO_K^{\times}\), then \(\kk(S_{\frakp_a}) \simeq \kk(S_{\frakp_{a'}})\) if and only if \(a'a^{-1}\in \OO_{K'}^{\times e}\).
            In particular, there are \([\OO_K^{\times} : \OO_K^{\times}\cap \OO_{K'}^{\times e}]\) 
            isomorphism classes of totally tamely ramified degree $e$ extensions of $K'$ that can arise as $\kk(S_{\frakp_a})$ for some $a \in \OO_K^{\times}$; and for every $L/K'$ that arises, 
            \[
                \mu_{\OO_K}(\{ a \in \OO_K^{\times} : 
                \kk(S_{\frakp_a})\simeq L\}) = \mu_{\OO_K}(\OO_K^{\times})\mu_{\OO_K^{\times}}(\OO_K^{\times}\cap \OO_{K'}^{\times e})=
                \frac{q_K-1}{q_K}\cdot\frac1{[\OO_K^{\times} : \OO_K^{\times}\cap \OO_{K'}^{\times e}]}.
            \]
        \end{enumerate}
    \end{enumerate}
\end{prop}
\begin{remark}\label{rmks:RamifiedLocalPresentation}\hfill
    The statements in part~\eqref{part:CurveRamIsom} are independent of the choice of uniformizer of \(K\).
\end{remark}
The following lemma will be useful in the proof.
\begin{lemma}\label{lem:LiftingFactorizations}
    Let \(K\) be a nonarchimedean local field and let \(g_1,g_2, \tilde{g_1}, \tilde{g_2}\in \OO_{K}[z]\) be monic polynomials such that \(\deg(g_i) = \deg(\tilde{g}_i)\), \(g_i\equiv \tilde{g}_i\bmod \pi\) for \(i=1,2\)  and \(\gcd(g_1\bmod \pi,g_2\bmod \pi)= 1\).
    If \(g_1g_2 \equiv \tilde{g_1}\tilde{g_2}\bmod \pi^v\) for some positive integer \(v\), then 
    \(
    g_i\equiv \tilde{g_i}\bmod \pi^v
    \) for \(i = 1,2.\)
\end{lemma}
\begin{proof}
    We prove this by induction; the base case is given by assumption. Fix \(1\leq n < v\) and assume that \(\tilde{g_i}\equiv g_i\bmod \pi^n\), i.e., that \(\tilde{g_i} = g_i + \pi^n m_i \) for some polynomials \(m_i\) with \(\deg(m_i) < \deg(g_i)\). Then
    \(
    \tilde{g_1}\tilde{g_2} = g_1g_2 + \pi^n(m_1g_2 + m_2g_1) + \pi^{2n}m_1m_2 \equiv g_1g_2\pmod{\pi^v}.
    \)
    Since \(v>n\geq1\), this implies that \(m_1g_2 + m_2g_1\equiv 0\bmod \pi\). By assumption \(g_1\bmod \pi\) and \(g_2\bmod \pi\) are relatively prime, so \(g_i|m_i\) modulo \(\pi\). Since \(\deg(m_i)<\deg(g_i)\), we must have \(m_i\equiv 0 \bmod \pi\).    
\end{proof}

\begin{proof}[Proof of Proposition~\ref{prop:RamifiedLocalPresentation}]
    Let \(\frakp\subset R\) be a height \(1\) prime (hence principal) such that \(\OO_K\simeq R/\frakp\). Thus, there exists an \(a'\in \OO_K\) such that \(-a'\equiv p_0 \bmod \frakp\), i.e., \(p_0 + a'= pr\) for some \(r\in R\) and generator \(p\) of \(\frakp\). 
    Since \(R/\frakp_0\) is also isomorphic to \(\OO_K\), there exists an \(a_0\in \OO_K\) such that \(p - a_0 = p_0r_0\) for some \(r_0\in R\). Hence, 
    \[
        prr_0 = p_0r_0 + a'r_0 = p - a_0 + a'r_0, \quad \textup{and} \quad
    p_0rr_0 = pr - a_0r = p_0 + a' - a_0r
    \]
    so \(a_0 = a'\cdot(r_0\bmod \pp)\in \OO_K\) and \(a' = a_0\cdot(r\bmod \pp_0)\in \OO_K\).  In particular, \(r_0\bmod \pp \) and \(r\bmod \pp_0\) are units in \(\OO_K\). Since \(R\) is local, this implies that \(r, r_0\in R^{\times}\) and so \(\frakp = \langle pr\rangle = \langle p_0 + a'\rangle\). Moreover, since \(\frakp + \frakp_0\subset \mm_R\), we have \(a'\in \pi\OO_K\), i.e., \(\frakp = \frakp_a:= \langle p_0 + a\pi\rangle\) for some \(a\in \OO_K\) as desired. 

    By Lemma~\ref{lem:MonogenicLocal}, \(S = R[\beta] \simeq R[z]/f(z)\), where \(f(z)\) is the minimal polynomial of \(\beta\). 
    Since \(R\to S\) is not \'etale and the assumptions of Lemma~\ref{lem:RegularLocalRingPresentation} are satisfied, there is a monic polynomial \(\tilde{f}\in \OO_K[z]\), irreducible modulo \(\pi\), such that \(f - \tilde{f}^e = p_0 s\) for some \(e>1\), \(s\in R[z]\) whose image in \(S\) is a unit. From this presentation, we can conclude 
    part~\eqref{part:Dimension}. Since \(f \bmod \frakm_R = (f\bmod \frakp_a)\bmod \pi \equiv \tilde{f}^e\), the field \(L\) contains the unramified extension of degree equal to \(\deg(\tilde{f})\).  Moreover, since \(\frakm_S = \frakm_R S + \langle\tilde{f}\rangle\), we conclude \([S/\mm_S:R/\mm_R] = \deg(\tilde{f})\), proving~\eqref{part:Inertia}.

    Now we consider part~\eqref{part:Ramification}; by part~\eqref{part:Inertia}, \(L\) contains \(K':=K([S/\mm_S:R/\mm_R])\) so it suffices to compute \(e(L/K')\). Let \(\theta_i\in K'\) denote the roots of \(\tilde{f}\). Recall the relation of \(f\) given in Lemma~\ref{lem:RegularLocalRingPresentation}, i.e., \(f= \tilde{f}^e + p_0s\) where \(e = e(\mm_S/\mm_R)\). Note that \(s\) depends only on \(f\) and \(\frakp_0\); it does not depend on \(\frakp_a\). Then we have
    \begin{equation}\label{eq:fmodulop}
        f - \tilde{f}^e = p_0s = (p - a\pi)s. 
    \end{equation}
    Since \(f\bmod \frakp_a\) has the same reduction modulo \(\pi\) as \(\tilde{f}^e = \prod_i(z - \theta_i)^e\), Hensel's lemma implies that \(f(z)\bmod \frakp_a\) factors as \(h_1h_2\dots h_r\), where each \(h_i\in \OO_{K'}[z]\) is monic and \(h_i\equiv (z - \theta_i)^e\bmod \pi\). In particular, \(L\) is obtained from \(K'\) by adjoining the root of a degree \(e\) polynomial, so \(e(L/K')\leq e = e(\mm_S/\mm_R) \).

    Setting \(v := 1 + v(a)\) and \(u:= a\pi^{1-v} \in \OO_K^{\times}\), the above equalities and congruences yield:
    \[
        h_1(z)\cdots h_r(z) \equiv f\bmod \frakp_a = \tilde{f}^e - u\pi^{v} s = (z- \theta_1)^e\cdots (z - \theta_r)^e - u\pi^v s.
    \]
    Since for \(i\neq j\), \(h_i\equiv (z - \theta_i)^e\pmod \pi\) is relatively prime to \(h_j\equiv (z - \theta_j)^e\pmod \pi\) and \(h_i,h_j, (z-\theta_i)^e, (z-\theta_j)^e\) are all monic of degree \(e\), Lemma~\ref{lem:LiftingFactorizations} implies that \(h_i(z) \equiv (z - \theta_i)^e\pmod{\pi^v}\) for all \(i\). In particular, if we set \(\theta := \theta_1\) and \(g(z):= h_1(z + \theta_1)\), then \(g(z) \equiv z^e\bmod \pi^v\) in \(R/\frakp_a[\theta, z] \simeq\OO_{K'}[z]\).  Furthermore, by \eqref{eq:fmodulop}
    \begin{equation}
        g(0) = (f\bmod\pp_a)(\theta)\prod_{i=2}^rh_i(\theta)^{-1} = -u\pi^vs(\theta)\prod_{i=2}^rh_i(\theta)^{-1} 
        \equiv -u\pi^vs(\theta)\prod_{i=2}^r(\theta - \theta_i)^{-e}\bmod \pi^{v+1}\label{eq:g0}
    \end{equation}
    Since \(s\in R[z]\) gives a unit in \(S\), \(s(\theta)\in \OO_{K'}^{\times}\). Thus, the Newton polygon of \(g\) has a single segment of slope \(-\frac{v}{e}\), which gives the desired lower bound in~\eqref{part:Ramification}.
     
    Furthermore, if \(e(L/K) = e\) (which is implied if \(v\) is relatively prime to \(e\)), then combining~\eqref{part:Dimension} and~\eqref{part:Inertia} we have
    \[
        e\cdot [S/\mm_S: R/\mm_R] = \dim_K \kk(S_{\pp_a}) \geq \dim_K L = e(L/K)f(L/K) \geq e \cdot [S/\mm_S:R/\mm_R].
    \]
    Thus, \(L = \kk(S_{\pp_a})\) and \(f(L/K) = [S/\mm_S:R/\mm_R]\), which proves~\eqref{part:DistancefrakptoB}.

    Finally, we prove part~\eqref{part:RamIsom}. Retain all the notation from above and assume that \(q_K\) is coprime to $e$; in particular, by~\eqref{part:Ramification}, \(L/K\) is tamely ramified. Since the maximal unramified subextension of \(\kk(S_{\frakp_a})/K\) is \(K'/K\), it suffices to show that there is a subgroup \(H\) that characterizes the isomorphism classes of totally tamely ramified field extensions \(\kk(S_{\frakp_a})/K'\). We claim this subgroup is \(H := \langle s(\theta) \rangle\OO_K^{\times}\OO_{K'}^{\times e}/\OO_{K'}^{\times e}\).
    
    By Proposition~\ref{prop:IsoTotTamelyRam}, the isomorphism class of \(\kk(S_{\frakp_a})/K'\) is determined by a uniformizer in \(N_{\kk(S_{\frakp_a})/K'}(\kk(S_{\frakp_a})^{\times})\) modulo \(\OO_{K'}^{\times e}\). From the bijection in \eqref{eq:OK-primes}, we must understand how a uniformizer in the norm subgroup of \(\kk(S_{\frakp_a})/K'\) depends on \(a = u\pi^{v-1}\) (recall that \(\frakp_a = \langle p_0 + a\pi\rangle\)). 
    We computed above that \(\kk(S_{\frakp_a})/K'\) is generated by adjoining a root of \(g(z)\) and thus \((-1)^eg(0)\in N_{\kk(S_{\frakp_a})/K'}(\kk(S_{\frakp_a})^{\times})\). Since \(K'^{\times e}\) is contained in the norm subgroup,~\eqref{eq:g0} shows that the norm subgroup contains the uniformizer
    \(
        \left(-us(\theta)\right)^{m_v} \pi,
    \)
    where \(m_v\) is a multiplicative inverse of \(v\) modulo \(e\).  
    
    Let us consider~\eqref{part:CurvePartitionByv}. If \(a,a'\in \OO_{K}^{\times}\) (so \(u = a\) and \(u' = a'\)), then \(-us(\theta)\equiv -u's(\theta)\bmod \OO_{K'}^{\times e}\) if and only if \(u'u^{-1}\in \OO_{K'}^{\times e}\). Thus, \(\kk(S_{\frakp_a}) \simeq \kk(S_{\frakp_{a'}})\) if and only if \(a'a^{-1}\in \OO_{K'}^{\times e}\). In particular, the totally tamely ramified degree \(e\) extensions of \(K'\) that can appear as \(\kk(S_{\frakp_a})\) for some \(a\in \OO_K^{\times}\) are in bijection with \(\OO_K^{\times}/\OO_K^{\times}\cap \OO_{K'}^{\times e}\). 
    Moreover, for a fixed \(a_0\in \OO_K^{\times}\),
    \begin{align*}
        \mu_{\OO_K}\left(\left\{a\in \OO_K^{\times} : \;\kk(S_{\frakp_a})\simeq \kk(S_{\frakp_{a_0}}) \right\}\right) & =  \mu_{\OO_K}\left(\left\{a\in \OO_K^{\times} : aa_0^{-1}\in \OO_{K'}^{\times e} \right\}\right) \\& = \mu_{\OO_K}(\OO_K^{\times})\mu_{\OO_K^{\times}}\left(a_0 (\OO_K^{\times}\cap\OO_{K'}^{\times e}) \right)\\
        & = \mu_{\OO_K}(\OO_K^{\times})\mu_{\OO_K^{\times}}\left(\OO_K^{\times}\cap\OO_{K'}^{\times e} \right)\\
        & = \frac{q_K-1}{q_K}\cdot\frac1{[\OO_K^{\times} : \OO_K^{\times}\cap \OO_{K'}^{\times e}]}.
    \end{align*}
    
    To prove the rest of~\eqref{part:RamIsom} it remains to prove the following statements:
    \begin{enumerate}
        \item[\eqref{part:Existence}] If \(\tilde{u}\in H\) generates \(H/\OO_{K}^{\times}\OO_{K'}^{\times e}\), then there exists a \(u\in \OO_K^{\times}\) and \(v\in \mathbb{N}\) such that \(\left(-u s(\theta)\right)^{m_{v}}\tilde{u}^{-1}\in \OO_{K'}^{\times e}\).
        
        \item[\eqref{part:Uniqueness}] If \(u\in \OO_K^{\times}\) and \(v\in \mathbb{N}\), then \(\langle\left(-us(\theta)\right)^{m_v}\rangle\) generates \(H/\OO_{K}^{\times}\OO_{K'}^{\times e}\).
    \end{enumerate}
    Let \(\tilde{u}\in H\) be an element that generates \(H/\OO_K^{\times}\OO_{K'}^{\times e}\), i.e., \(\tilde{u}^v = -u s(\theta)u'^e\) for some integer \(v\) that is coprime to \(e\), some \(u\in \OO_{K}^{\times}\), and some \(u'\in \OO_{K'}^{\times}\). Taking the \(m_v^{th}\) power of both sides yields~\eqref{part:Existence}.

    Let \(u\in \OO_K^{\times}\) and \(v\in \mathbb{N}\). Since \(m_v\) is an inverse of \(v\) modulo \(e\), \(\langle\left(-us(\theta)\right)^{m_v}\rangle = \langle-us(\theta)\rangle \subset \OO_{K'}^{\times}\OO_{K'}^{\times e}/\OO_{K'}^{\times e}\), and this in turn equals \(H\) modulo \(\OO_{K}^{\times}\OO_{K'}^{\times e}\). Thus we have~\eqref{part:Uniqueness}.
\end{proof}

\section{Isomorphism classes of fibers over a nonarchimedean local field}\label{sec:LocalIso}\label{sec:CoversOfCurves}

Let \(\psi\colon C \to D\) be a nonconstant morphism of smooth projective geometrically integral curves over \(K\). We are interested in morphisms with \defi{good reduction}, which morally means there is an integral model \(\Psi\colon \calC \to \calD\) 
such that the combinatorial ramification data of \(\psi\) agrees with the combinatorial ramification data of \(\Psi_{\F}\colon \calC_{\F}\to \calD_{\F}\).  In other words, the ramification of \(\Psi_{\F}\colon \calC_{\F}\to \calD_{\F}\) is all caused by geometric ramification of \(\psi\).

\begin{defn}\label{def:GoodReductionOfCurves} We say that the morphism \(\psi\) has \defi{good reduction} if:
    \begin{enumerate}
        \item There exist smooth proper \(\OO_K\)-schemes \(\calC, \calD\) whose generic fibers \(\calC_K, \calD_K\) are isomorphic to \(C, D\), respectively, \label{gr:smoothmodel}
        \item There exists a finite surjective
        morphism \(\Psi\colon \calC \to \calD\) such that \(\Psi_{K}=\psi\), and\label{gr:flatmorphism}
        \item There exists a minimal reduced closed codimension \(1\) subscheme \(\calZ\subset \calC\) such that \(\Psi|_{\calC\smallsetminus \calZ}\colon \calC\smallsetminus \calZ \to \calD\) is \'etale, and for this minimal subscheme, the maps \(\calZ \to \Psi(\calZ)\) and \(\Psi(\calZ) \to \Spec \OO_K\) are both \'etale.\label{gr:etaleramification}
    \end{enumerate}
\end{defn}
If \(\psi\) has good reduction, then~\eqref{gr:etaleramification} implies that \(\calZ\) and \(\Psi(\calZ)\) are \'etale covers of \(\Spec \OO_K\) so by the classification statement mentioned in the introduction, \(\calZ\) and \(\Psi(\calZ)\) are isomorphic to \(\Spec\prod_{i}\OO_{K(d_i)}\) for not necessarily equal tuples of positive integers \(d_i\). In other words,
for every closed point \(\overline{x}\in \calC\) the extension of local rings \(\OO_{\calD, \Psi(\overline{x})}\hookrightarrow\OO_{\calC, \overline{x}}\) has at most one codimension \(1\) point causing ramification (see Lemma~\ref{lem:GoodReductionConseq} for the details). Thus we may use the results of Section~\ref{sec:LocalRingExtns} to study the isomorphism classes of fibers \(C_t\). To precisely state our results, we introduce some notation.

Given such a morphism of good reduction, we write \(\calB\subset \calD\) for the (reduced) branch locus of \(\Psi\) and \(B\) for the generic fiber \(\calB_K\) (note that \(\calB\) agrees with the reduced image \(\Psi(\calZ)\) of the subscheme \(\calZ\) given by the definition). We will freely use the equality \(B(K) = \calB(\OO_K)\) given by the valuative criterion for properness. 

For \(t\in (D\smallsetminus B)(K)\), the fiber \(C_t\) is reduced and so \(\kk(C_t)\) is a product of fields \(\prod_{x\in C_t} \kk(x)\). 
As mentioned in the introduction, the maximal subfields of \(\kk(C_t)\) are partitioned by their specializations, i.e.,
\begin{equation}\label{eqn:EtaleAlgDecomposition}
    \kk(C_t) = \prod_{\overline{x}\in \calC_{\overline{t}}} \kk(C_t)_{\overline{x}}, \quad \textup{where} \quad
    \kk(C_t)_{\overline{x}} := \prod_{\substack{x\in C_t\\x \rightsquigarrow \overline{x} }} \kk(x). 
\end{equation}
The main result of this section determines the structure of these \'etale algebras \(\kk(C_t)_{\overline{x}}\) in terms of the geometry of the map \(\Psi\) around \(\overline{x}\).
\begin{thm}\label{thm:RamAndInertiaDegrees}
    Let \(\psi\colon C \to D\) be a morphism between smooth projective geometrically integral curves that has good reduction,  let \(t\in (D-B)(K)\), let $\overline{t}$ be the specialization of $t$, and let \(\overline{x}\in \calC_{\overline{t}}\). Define
    \[
        v(t, \calB) := \begin{cases}
            \max_{t_0\in B(K)}\left\{n : t= t_0 \in \calD (\OO_K/\pi^n)\right\} & \textup{if }e(\overline{x}/\overline{t}) > 1, \\
            \infty & \textup{otherwise}.
        \end{cases}
    \]
    \noindent Then: 
    \begin{enumerate}
        \item \(\dim_K \kk(C_t)_{\overline{x}} = e({\overline{x}/\overline{t}})\deg(\overline{x})\).\label{it:dim}
        \item Any maximal subfield \(L\subset\kk(C_t)_{\overline{x}}\) has inertia degree \(f(L/K)\geq\deg(\overline{x})\).
        \item Any maximal subfield \(L\subset\kk(C_t)_{\overline{x}}\) has ramification degree 
        \[
        \frac{e(\overline{x}/\overline{t})}{\gcd(v(t, \calB), e(\overline{x}/\overline{t}))}\leq e(L/K)\leq e(\overline{x}/\overline{t}).
        \]
        \item If \(v(t,\calB)\) is relatively prime to \(e(\overline{x}/\overline{t})\) (or, more generally, if \(e(L/K) = e(\overline{x}/\overline{t})\)), then \(\kk(C_t)_{\overline{x}}\) is a field with ramification degree equal to \(e(\overline{x}/\overline{t})\) and inertia degree equal to \(\deg(\overline{x})\).\label{it:RelPrimeCase}
        \item \label{part:CurveRamIsom} 
        Assume that \(e: = e(\overline{x}/\overline{t})\) is coprime to \(q_K\), which by~\eqref{part:Ramification} implies that all maximal subfields \(L\subset \kk(C_t)_{\overline{x}}\) are tamely ramified over \(K\). Let \(K' := K(\deg \overline{x})\) and fix a uniformizer \(\pi\) of \(K\).
        Then there exists a subgroup \(\OO_{K}^{\times}\OO_{K'}^{\times e}/\OO_{K'}^{\times e}<H< \OO_{K'}^{\times}/\OO_{K'}^{\times e}\) (necessarily cyclic) that completely characterizes the isomorphism classes of field extensions
        \(\kk(C_t)_{\overline{x}}\) for \(t\) with \(v(t,\calB)\) coprime to \(e(\overline{x}/\overline{t})\).
        More precisely:
        \begin{enumerate}
            \item \label{part:CurveExistence}
            If \(u\in \OO_{K'}^\times\) generates \(H/\OO_K^{\times}\OO_{K'}^{\times e}\) and \(L/K'\) is a totally tamely ramified extension of degree \(e(\overline{x}/\overline{t})\) with \(u\pi\in N_{L/K}(L^{\times})\), then there exists a \(t \in (D\smallsetminus B)(K)\) with \(t\rightsquigarrow\overline{t}\) such that \(L\simeq \kk(C_t)_{\overline{x}}\).
            
            \item \label{part:CurveUniqueness}
            Conversely, if \(v(t,\calB)\) is relatively prime to \(e(\overline{x}/\overline{t})\) and \(\varpi_{\frakp}\in N_{\kk(C_t)_{\overline{x}}/K'}(\kk(C_t)_{\overline{x}}^{\times})\) is a uniformizer, then \(\varpi_{\frakp}\pi^{-1}\) generates \(H/\OO_K^{\times}\OO_{K'}^{\times e}\).

            \item \label{part:CurvePartitionByv} Furthermore, there are \([\OO_K^{\times} : \OO_K^{\times}\cap \OO_{K'}^{\times e}]\) 
            isomorphism classes of totally tamely ramified degree $e$ extensions of $K'$ that can arise as $\kk(C_{t_0})_{\overline{x}}$ for some $t_0 \in \OO_K^{\times}$; and for every such $\kk(C_{t_0})_{\overline{x}}$, 
            then
            \[
            \mu_{D}\left(\left\{t\in (D\smallsetminus B)(K) : t\rightsquigarrow\overline{t_0}, \; v(t,\calB) = 1, \;\kk(C_t)_{\overline{x}}\simeq \kk(C_{t_0})_{\overline{x}}\right\}\right) = \frac{q_K-1}{q_K}
            \cdot\frac1{[\OO_K^{\times} : \OO_K^{\times}\cap \OO_{K'}^{\times e}]}.
            \]
        \end{enumerate}
    \end{enumerate}
    \end{thm}

    The unramified case of this result (i.e., the case when \(e(\overline{x}/\overline{t}) = 1\)) follows by a standard argument using the classification of \'etale extensions of \(\OO_K\); we review this in Section~\ref{sec:UnramifiedCase}. The main content of the theorem is in the ramified case. To prove this result, we first articulate some consequences of the good reduction hypothesis (Section~\ref{sec:GoodReductionConsequences}). In particular, we show that local ring extensions coming from these covers of curves satisfy the assumptions of Lemmas~\ref{lem:MonogenicLocal} and~\ref{lem:RegularLocalRingPresentation}, so we may apply Proposition~\ref{prop:RamifiedLocalPresentation} to conclude Theorem~\ref{thm:RamAndInertiaDegrees} (see  Section~\ref{sec:ProofOfRamAndInertiaDegrees} for details).

    \subsection{The unramified case}\label{sec:UnramifiedCase}
    The methods in this section are standard; we include the details for the convenience of the reader.
\begin{prop}[Special case of Theorem~\ref{thm:RamAndInertiaDegrees}]\label{prop:EtAlgUnramified}
    Let \(\overline{t}\in \calD\smallsetminus\calB\) be a closed point so \(\kk(\calC_{\overline{t}})\) is isomorphic to a product of \(r\) fields and let \(d_i\) denote their degree over \(\kk(\overline{t})\) (for \(i= 1, \dots r\)). Then for any codimension \(1\) point \(t \subset D\) that specializes to \(\overline{t}\), we have \[\kk(C_t)\simeq \kk(t)(d_1) \times \kk(t)(d_2) \times \dots \times \kk(t)(d_r),\] where \(\kk(t)(d_i)\) denotes the unramified degree \(d_i\) extension of \(\kk(t)\). 
\end{prop}
\begin{proof}[Direct proof of Proposition~\ref{prop:EtAlgUnramified}]
    Let \(\calU = \calD\smallsetminus \calB\), and let \(\mathcal{T}\subset \calD\) denote the Zariski closure of \(t\). 
    By assumption \(\calT\subset \calU\), so \(\calC_{\calT}\to \calT\simeq \Spec \OO\) is the base change of the \'etale map \(\calC_\calU \to \calU\) and so is \'etale. Therefore \(\calC_{\calT}\) is isomorphic to a disjoint union of \(\Spec \OO'\), where \(\Frac(\OO')/K\) is an unramified extension. The description of the special fiber \(\calC_{\overline{t}}\) then implies that
    \(\kk(C_t)\simeq \kk(t)(d_1) \times \kk(t)(d_2) \times \dots \times \kk(t)(d_r).\)
\end{proof}

    \subsubsection{Determining inertia degrees from the cycle type of Frobenii}\label{ss:cycletypeFrob}
        Proposition~\ref{prop:EtAlgUnramified} shows that for points of \(\calD\) over which \(\psi\) is \'etale, the isomorphism type of the fiber is completely determined by the special fiber of \(\Psi\), i.e., the map \(\Psi_{\F}\colon \calC_{\F} \to \calD_{\F}\); this in turn is controlled by Frobenii elements.

        \begin{prop}\label{prop:etalealg-nonGal}\label{prop:EffCheb}
            Let \(\psi_0 \colon C_0 \to D_0\) be a degree \(d\) morphism of smooth projective geometrically integral curves over a finite field \(\F\), let \(\widehat\psi_0\colon \widehat{C}_0\to D_0\) be the Galois closure of \(\psi_0\), let \(G\) be its Galois group and let $\overline{t} \in D$ such that \(\psi_0\) is \'etale at \(\overline{t}\).
            \begin{enumerate}
                \item \label{part:etale} The Galois closure \(\widehat\psi_0\) is also \'etale at \(\overline{t}\) so
            there is a Frobenius conjugacy class \(c_{\overline{t}}\subset S_{d}\).
            \item \label{part:Isom} If $(d_i)_{i=1}^r\vdash d$ is the cycle type of an element in \(c_{\overline{t}}\), then \[
                \kk(({C}_0)_{\overline{t}}) \simeq \prod_{i=1}^r \kk(\overline{t})(d_i).
            \]
            \item \label{part:EffCheb} Let \(U_0\subset D_0\) be the smooth locus of \(\psi_0\) and let \((d_i)\vdash d\) be a partition. Then
            \[
                \frac{\#\left\{\overline{t} \in U_0(\F) : \kk(({C}_0)_{\overline{t}})\simeq \prod \F(d_i)\right\}}{\#U_0(\F)} = \frac{\#\left\{\sigma\in G : \ct(\sigma) = (d_i)\right\}}{|G|} + O\left(|\F|^{-\frac{1}{2}}\right),
            \]
            where the implied constant in the $O$ depends on the genus of $\widehat{C}_0$, the genus of $D_0$, and the group \(G\).
            \end{enumerate}
        \end{prop}

	\begin{proof}
            Let \(\overline{y}\) be a point of \(\widehat{C}_0\) lying over \(\overline{t}\) and let \(I_{\overline{y}}\) denote its inertia subgroup. Since \(\psi_0\) is unramified at \(\overline{t}\), the inertia subgroup \(I_{\overline{y}}\) must be contained in \(G_0 := \Gal(\widehat{C}_0/C_0)\). Similarly, the same is true for \(I_{\overline{y}'}\) for any other point \(\overline{y}'\in \widehat{C}_0\) lying over \(\overline{t}\) and so the smallest normal subgroup \(N\) of \(G:=\Gal(\widehat{C}_0/D_0)\) that contains \(I_{\overline{y}}\) must be contained in \(G_0\). However, \(G_0\) is simple (since \(\widehat{C}_0\) is the Galois closure). Thus \(N=\{1\} = I_{\overline{y}}\), proving~\eqref{part:etale}.

            Let $c_{\overline{t}}$ denote the Frobenius conjugacy class (i.e., the conjugacy class of a generator of the decomposition group that maps to the Frobenius element), let $\sigma_{\overline{t}}$ denote an element in the Frobenius conjugacy class $c_{\overline{t}}$, and let $\overline{y} \in (\widehat{C}_0)_{\overline{t}}$ be the corresponding point (i.e., $\sigma_{\overline{t}}$ fixes $\overline{y}$).
		Then the \'etale algebra $\kk((\widehat{C}_0)_{\overline{t}})$ is
		\[
			\kk((\widehat{C}_0)_{\overline{t}})=
                \prod_{\overline{y}'\in (\widehat{C}_0)_{\overline{t}}} \kk({\overline{y}'}) =
                \prod_{g\langle{\sigma_{\overline{t}}}\rangle\in G/\langle{\sigma_{\overline{t}}}\rangle} \kk({g(\overline{y})}) 
                \simeq \Ind_{\langle \sigma_{\overline{t}} \rangle}^G \kk({\overline{y}}).
		\]
            By Galois theory \(\kk((C_0)_{\overline{t}})\) is the subfield of \(\kk((\widehat{C}_0)_{\overline{t}})\) fixed by \(G_0\), so we have:
		\[
			\kk((C_0)_{\overline{t}})= 
   \prod_{G_0g \langle{\sigma_{\overline{t}}}\rangle\in G_0 \backslash G /\langle \sigma_{\overline{t}} \rangle}\left(\prod_{g_0g\langle{\sigma_{\overline{t}}}\rangle\in G_0g \langle{\sigma_{\overline{t}}}\rangle/\langle{\sigma_{\overline{t}}}\rangle} \kk({g_0g(\overline{y})})\right)^{G_0} = 
   \prod_{g \in G_0 \backslash G /\langle \sigma_{\overline{t}} \rangle} \kk({g(\overline{y})})^{G_0 \cap g\langle \sigma_{\overline{t}} \rangle g^{-1}}. 
		\]
            It remains to determine the extension \(\kk({g(\overline{y})})^{G_0 \cap g\langle \sigma_{\overline{t}} \rangle g^{-1}}\) which, as a finite extension of the finite field \(\kk(\overline{t})\), is uniquely determined by its degree. Note that by Galois theory
            \[
            [\kk(g(\overline{y}))^{G_0 \cap g\langle \sigma_{\overline{t}}\rangle g^{-1}}:\kk(\overline{t})] = \frac{[\kk(g(\overline{y})):\kk(\overline{t})]}{|{G_0 \cap g\langle \sigma_{\overline{t}}\rangle g^{-1}}|} = \frac{|g\sigma_{\overline{t}}g^{-1}|}{|{G_0 \cap g\langle \sigma_{\overline{t}}\rangle g^{-1}}|} = |\langle g \sigma_{\overline{t}}  g^{-1}\rangle G_0/G_0|.
            \]
            The group-theoretic lemma below completes the proof of part~\eqref{part:Isom}.
Finally~\eqref{part:EffCheb} follows from~\eqref{part:Isom} and the Effective Chebotarev Density Theorem~\cite{Murty-Scherk}*{Thm.~1}.
	\end{proof}

    \begin{lemma}\label{lem:cycletype-grouptheory}
         Let $G < S_d$ be a transitive subgroup, let \(\sigma\in G\) be an element with cycle-type $(d_1, d_2, \ldots, d_r)$, and let \(G_0 := \textup{Stab}(d) < S_d\). Then there are \(r\) double cosets \(G_0\backslash G \slash \langle \sigma \rangle\), and, up to relabeling the double coset representatives \(g_1, \dots, g_r\), we have
        \[
            |\langle g_i \sigma  g_i^{-1}\rangle G_0/G_0| = d_i.
        \]
    \end{lemma}
    \begin{proof}
        For $g\in G$, let \(\ell(g)\) denote the order of the orbit of \(d\) under \(g\), i.e., in the cycle representation of \(g\), $d$ is in a cycle of length $\ell(g)$. Then \(g^n\in G_0\) if and only if $n$ is a multiple of $\ell(g)$, so $|\langle g \rangle G_0/G_0|=\ell(g)$. We label the cycles of $\sigma$ such that the $i$-th cycle has length $d_i$. By transitivity for each $1\leq i\leq r$, there is a $g_i\in G$ that conjugates the $i$-th cycle of $\sigma$ to a cycle containing $d$. Thus, by the argument above, \(|\langle g_i \sigma  g_i^{-1}\rangle G_0/G_0| = d_i\). 

        Now let \(g\in G\) and let \(i\) be such that \(g\) conjugates the cycle containing $d$ to the $i$-th cycle of \(\sigma\). In other words, there exists some integer \(j\) such that \(g\sigma^jg_i^{-1} \in G_0\). Thus \(g_1, \dots, g_r\) give a complete set of double coset representatives.
    \end{proof}
    \subsection{Consequences of good reduction}
    \label{sec:GoodReductionConsequences}

    \begin{lemma}\label{lem:GoodReductionConseq}
        Let \(\psi\colon C \to D\) be a morphism of good reduction,  let \(\Psi\colon\calC\to \calD\) be the associated integral model, and let \(\psi_0\colon C_0 \to D_0\) denote the restriction of \(\Psi\) to the special fiber. Let \(\overline{x}\in \calC\) be a closed point and let \(\overline{t} := \Psi(\overline{x})\). We have the following properties.
        \begin{enumerate}
            \item The ring map \(R := \OO_{\calD, \overline{t}}\hookrightarrow S := \OO_{\calC, \overline{x}}\) is a finite extension of regular local rings.
            \item If \(e(\overline{x}/\overline{t})=1\), then \(R\hookrightarrow S\) is \'etale.
            \item If \(e(\overline{x}/\overline{t})>1\), then there exists a unique height \(1\) prime ideal \(\frakp_0\subset R\) such that \(R\to S\) is \'etale at all primes \(\frakq\subset S\) such that \(\frakq\cap R\) does not contain \(\frakp_0\). Furthermore, for this prime \(\frakp_0\):
            \begin{enumerate}
                \item there is a unique prime \(\frakq_0\subset S\) lying over \(\frakp_0\),\label{property:unique}
                \item the ring maps \(R/\frakp_0 \to S/\frakq_0\) and \(\OO_K\to R/\frakp_0\) are both \'etale and \([R/\pp_0:\OO_K] = \deg(\overline{t})\).\label{property:etale}
            \end{enumerate}
            In other words, there is a unique codimension \(1\) point \(x\in \calC\) that specializes to \(\overline{x}\) where \(\Psi\) is ramified and, for this point,
            \[
            e(x/\psi(x)) = e(\overline{x}/\overline{t})\quad \textup{and}\quad\deg(x) = \deg(\overline{x}).
            \]
        \end{enumerate}
    \end{lemma}
    \begin{remark}
        The lemma in particular implies that if \(\overline{x}\in \calC\) is a closed point with \(e(\overline{x}/\overline{t})>1\), then the extension  \(\OO_{\calD,\overline{t}} \to \OO_{\calC, \overline{x}}\) satisfies the assumptions of Lemmas~\ref{lem:MonogenicLocal} and~\ref{lem:RegularLocalRingPresentation}.
    \end{remark}
    \begin{proof}
        The rings \(R := \OO_{\calD, \overline{t}}\) and \(S := \OO_{\calC, \overline{x}}\) are regular local rings because (by Definition~\ref{def:GoodReductionOfCurves}\eqref{gr:smoothmodel}) \(\calC\) and \(\calD\) are smooth \(\OO_K\)-schemes.  
        Definition~\ref{def:GoodReductionOfCurves}\eqref{gr:flatmorphism} implies that \(\Psi\) induces a finite extension of rings \(R\hookrightarrow S\).  

        If \(e(\overline{x}/\overline{t}) = 1\), then \(\mm_R S = \mm_S\). Since any extension of finite fields is separable, \(R\to S\) is \'etale at \(\mm_S\). Since \(S\) is local, this implies that \(R\to S\) is \'etale at all primes.

        Now assume that \(e(\overline{x}/\overline{t})>1\). This implies that \(\overline{x}\in\calZ\) so by purity of the ramification locus, \(\calZ\cap \Spec S\) is cut out by a nonzero ideal \(I\) that is contained in some height \(1\) prime ideal.  
        By Definition~\ref{def:GoodReductionOfCurves}\eqref{gr:etaleramification}, \(\OO_K \to R/I\cap R\) and \(R/I\cap R \to S/I\) are both \'etale. Thus \(S/I\) is an \'etale local extension of \(\OO_K\) and hence is isomorphic to \(\OO_{K(d)}\) for some \(d\). In particular, \(I\) is a prime ideal \(\frakq_0\), or, equivalently, there is a 
        unique codimension \(1\) point \(x\in \calC\) that specializes to \(\overline{x}\) for which \(\psi\) is ramified.
        
        Set \(\frakp_0 := \frakq_0 \cap R\). By the above paragraph, we have \(\OO_K\to R/\frakp_0\) and \(R/\frakp_0 \to S/\frakq_0\) are \'etale. In particular, \(R/\frakp_0\) is the maximal order of an unramified field extension of \(K\) and thus, \([R/\frakp_0:\OO_K] = \deg(\overline{t})\). Thus, we have~\eqref{property:etale}. 
        
        It remains to prove that \(\frakp_0 S\) is contained in a unique height \(1\) prime ideal (necessarily \(\frakq_0\)). By Definition~\ref{def:GoodReductionOfCurves}\eqref{gr:flatmorphism}, \(S\) is a flat \(R\)-module (see, e.g.,~\cite{QingLiu}*{Chapter 4, Remark 3.11}), hence free. Thus,
        \begin{align*}
            [S/\mm_S:R/\mm_R]e(\mm_S/\mm_R) & = \dim_{R/\mm_R}S/\mm_R S = \rank_R S = \dim_{\kk(R_{\frakp_0})}S_{\frakp_0}/\frakp_0 S_{\frakp_0}\\
            & = \sum_{\frakq |\frakp_0} \dim_{\kk(R_{\frakp_0})}S_{\frakq}/\frakp_0 S_{\frakq}
            = \sum_{\frakq|\frakp_0}[S_{\frakq}/\frakq:R_{\frakp_0}/\frakp_0]e(\frakq S_{\frakq}/\frakp_0R_{\frakp_0}).
        \end{align*}
        By definition of \(\frakq_0\), if \(\frakq|\frakp_0\) and \(\frakq\neq \frakq_0\) then \(R\to S\) is \'etale at \(\frakq\), hence \(e(\frakq S_{\frakq}/\frakp_0R_{\frakp_0}) = 1\) and \([S_{\frakq}/\frakq:R_{\frakp_0}/\frakp_0] = [S/\mm_S:R/\mm_R]\). Moreover, \(\OO_K \to S/\frakq_0\) is \'etale, so \([S_{\frakq_0}/\frakq_0:R_{\frakp_0}/\frakp_0] = [S/\mm_S:R/\mm_R]\). Thus the above equation simplifies to
        \[
            e(\mm_S/\mm_R) 
            = e(\frakq_0 S_{\frakq_0}/\frakp_0R_{\frakp_0}) + \#\{\frakq|\frakp_0 : \frakq\neq \frakq_0 \}\geq e(\frakq_0 S_{\frakq_0}/\frakp_0R_{\frakp_0}).
        \]
        Geometrically, we may phrase this as \(e(x/t)\leq e(\overline{x}/\overline{t})\) and \(\deg(x) = \deg(\overline{x})\) where \(x\in \calC\) is the unique codimension \(1\) point that specializes to \(\overline{x}\) and is ramified. It remains to prove that \(e(x/t) = e(\overline{x}/\overline{t}).\)

        Now we work globally, i.e., considering the entire curve. The above relations give the inequality:
        \[
            \sum_{\overline{x}\in C_0} \sum_{\substack{x\in C\\ x\rightsquigarrow \overline{x}}} (e(x/t) - 1)\deg(x) \leq \sum_{\overline{x}\in C_0} 
            (e(\overline{x}/\overline{t}) - 1)\deg(\overline{x}).
        \]
        The left-hand side is the degree of the ramification divisor of \(\psi\) and the right-hand side is the degree of the ramification divisor of \(\psi_0\). By the Riemann-Hurwitz equation, these degrees are equal to \(2g(C) - 2 + 2\deg(\psi)\) and \(2g(C_0) - 2 + 2\deg(\psi_0)\), respectively.
        Since \(\calC \to \Spec \OO_K\) is smooth, \(g(C) = g(C_0)\) and since \(\Psi\) is finite \(\deg(\psi) = \deg(\psi_0)\). Thus we obtain the desired equality of ramification indices.
    \end{proof}

    \begin{cor}[Corollary of Lemmas~\ref{lem:GoodReductionConseq} and~\ref{lem:MonogenicLocal}]\label{cor:LocalStructure}
        Let \(\psi\colon C \to D\) be a morphism of good reduction,  let \(\Psi\colon\calC\to \calD\) be the associated integral model, let $\overline{t}\in \calD$ be a closed point, and let $\overline{x}\in \calC_{\overline{t}}$.  There exists a monic irreducible polynomial $f_{\overline{x}}(z)\in \OO_{\calD, \overline{t}}[z]$ such that \(\OO_{\calC, \overline{x}}\simeq \OO_{\calD, \overline{t}}[z]/f_{\overline{x}}(z)\).  In particular, for all $t\in D \smallsetminus B$ that specialize to $\overline{t}$, we have
        \[
                \kk(C_t)_{\overline{x}} \simeq \frac{\kk(t)[z]}{\langle f_{\overline{x}}(z) \bmod \pp_t\rangle}, 
        \]
        where $\frakp_t$ is the prime ideal corresponding to $t$ of the local ring $\calO_{\calD, \overline{t}}$.
    \end{cor}

    \subsection{Proof of Theorem~\ref{thm:RamAndInertiaDegrees}}\label{sec:ProofOfRamAndInertiaDegrees}
        Let \(R := \OO_{\calD,\overline{t}}\).  By Corollary~\ref{cor:LocalStructure}, there exists a monic irreducible polynomial $f_{\overline{x}}\in R[z]$ such that  \(\OO_{\calC, \overline{x}}\simeq R[z]/f_{\overline{x}}\), and 
        \[
                \kk(C_t)_{\overline{x}} \simeq \frac{\kk(t)[z]}{\langle f_{\overline{x}}(z) \bmod \pp_t\rangle}. 
        \]
        If \(e(\overline{x}/\overline{t}) = 1\), then \(f_{\overline{x}}\bmod \mm_R\) is a separable irreducible polynomial of degree equal to \(\deg f_{\overline{x}}\). Thus, for any prime \(\frakp\subset R\), \(f_{\overline{x}}\bmod \frakp\) is irreducible and so \(\kk(C_t)_{\overline{x}}\) is an unramified field extension of \(K\) of degree equal to \(\deg(f_{\overline{x}}) = \deg(f_{\overline{x}}\bmod \mm_R) = \deg(\overline{x})\), as desired. (Alternatively, see Section~\ref{sec:UnramifiedCase}.)

        Now assume that \(e(\overline{x}/\overline{t})>1\). Then \(R\to S\) is not \'etale. By Lemma~\ref{lem:GoodReductionConseq}, there exists a height one prime \(\frakp_0\subset R\) such that: there is a unique prime \(\frakq_0\) lying over \(\frakp_0\), \(R\to S\) is \'etale at all primes not containing \(\frakq_0\), and the maps \(\OO_K\to R/\frakp_0\) and \(R/\frakp_0\to S/\frakq_0\) are \'etale.
        Since \(R/\mm_R = \kk(\overline{t}) = \F\), this implies that \(\OO_K \to R/\frakp_0\) is an isomorphism. Moreover, this implies that there is a unique \(t_0\in B(K)\) such that \(t_0 \bmod \pi\equiv \overline{t}\). Thus,
        we may apply Proposition~\ref{prop:RamifiedLocalPresentation} with \(\frakp = \frakp_a\) the height one prime of \(R\) corresponding to \(t\). Furthermore, we have the following equalities: 
        \[
            \kk(S_{\frakp})  = \kk(C_t)_{\overline{x}},\;
            \deg(\overline{x}) = [S/\mm_S:R/\mm_R],\;
            e(\overline{x}/\overline{t}) = e(\mm_S/\mm_R), \;\textup{and}\;
            v(t, \calB)  = 1 + v(a).
        \]
        Therefore, Proposition~\ref{prop:RamifiedLocalPresentation} directly translates to Theorem~\ref{thm:RamAndInertiaDegrees}.
\qed

\section{Residue fields in a \(1\) dimensional linear system with specified local behavior}\label{sec:Global}

In this section we return to considering nonconstant morphisms \(\psi\colon C \to \PP^1\) of smooth projective geometrically integral curves defined over a number field \(k\). As mentioned in the introduction, there is a Zariski dense set of points \(t\in \PP^1(k)\) where  the fiber \(C_t\) is irreducible, which gives rise to infinitely many number fields \(\kk(C_t)\).  Conversely, the Riemann--Roch theorem tells us that all sufficiently high degree closed points on \(C\) arise as the fiber of some morphism \(\psi\).  (Moreover, if the Jacobian of \(C\) has rank \(0\), then one can strengthen ``all sufficiently high degree'' to ``all but finitely many''~\cite{BELOV}*{Theorem 4.2}; see~\cite{VirayVogt} for an expository treatment.) Thus, it is of interest to know which number fields arise as residue fields of irreducible fibers \(\kk(C_t)\).

The goal of this section is to review how the local results in Section~\ref{sec:CoversOfCurves} come together to give a global result as in Theorem~\ref{thm:SdIntro}. As mentioned in the introduction, these types of arguments appear throughout the literature and can be considered standard. We include complete proofs and statements for the reader's convenience.

In Section~\ref{subsec:preliminaries} we review how spreading out arguments give good reduction away from a finite set of primes and review an extension of Hilbert's irreducibility theorem that allows for the imposition of finitely many local conditions. In Section~\ref{sec:ChebotarevGlobal}, we show that the realizable unramified extensions are uniform across all but finitely many places and the realizable extensions are controlled by the Galois group of the Galois closure of \(\psi\) (Theorem~\ref{thm:ChebotarevMain}).  In Section~\ref{sec:PfSdIntro}, we combine these results with Theorem~\ref{thm:RamAndInertiaDegrees} to prove Theorem~\ref{thm:SdIntro}. This proof also serve as a template for the reader to assemble the general results in order to give a result analogous to Theorem~\ref{thm:SdIntro} for whichever class of morphisms \(C\to \PP^1\) is of interest to them. 

For the reader interested in finding explicit points \(t\in \PP^1(k)\) exhibiting specified local splitting behavior, we close in Section~\ref{sec:ModM} with some results showing that points of bounded height equidistribute among adelic open sets. This informs how far one must search to find a given local splitting type.

\subsection{Preliminaries}\label{subsec:preliminaries}

\begin{lemma}\label{lem:GoodReduction}
    Let \(C\) be a smooth projective geometrically integral curve over a number field \(k\) and let \(\psi\colon C \to \PP^1_k\) be a nonconstant morphism.
    Then there exists a finite set of places \(T\subset \Omega_k\) containing all archimedean places where \(\psi\) has good reduction for all \(v\in \Omega_k\smallsetminus T.\)

    If the Galois closure \(\widehat\psi\colon \widehat C \to \PP^1_k\) of \(\psi\) is geometrically integral, then we may enlarge \(T\) to assume that \(\hat \psi\) also has good reduction for all \(v\in \Omega_k\smallsetminus T.\) Then, for all \(v\in \Omega_k\smallsetminus T\), the special fiber of \(\widehat \calC\) is the Galois closure of the special fiber of \(\calC\to \PP^1_{\OO_v}\) with Galois group \(\Gal(\widehat{C}/\PP^1)\). 
\end{lemma}
\begin{proof}
This is a spreading out argument (see, e.g.,~\cite{PoonenRaPt}*{Section 3.2}). We provide the details in this situation for the reader's convenience.

We first spread out $C$ and \(\widehat C\) to schemes $\calC$ and $\widehat\calC$ over a dense open subscheme $\Spec \OO_{k, T}$ of $\Spec \OO_k$ as in \cite{PoonenRaPt}*{Theorem 3.2.1(i)}, so that the generic fiber  $\calC_k$ is isomorphic to $C$. By~\cite{PoonenRaPt}*{Theorem 3.2.1(ii) and Table 1, page 306-307}, which references~\citelist{\cite{EGA4-4}*{17.7.8(ii)} \cite{EGA4-3}*{9.7.7(iv) and 8.10.5(xii)}}\, we may further enlarge \(T\) and impose that \(\calC\to \Spec \OO_{k, T}\) and \(\widehat\calC\to \Spec \OO_{k, T}\) are smooth and proper
and that the fibers are geometrically integral.

Now we apply~\cite{PoonenRaPt}*{Theorem 3.2.1(iii)} to spread out $\psi$ and \(\widehat\psi\), so, after possibly enlarging \(T\), we obtain morphisms $\widehat\Psi\colon \widehat\calC\rightarrow \PP^1_{\OO_{k, T}}$ and $\Psi\colon \calC\rightarrow \PP^1_{\OO_{k, T}}$ such that \(\widehat\Psi\) factors through \(\Psi\). Further~\cite{PoonenRaPt}*{Theorem 3.2.1(iv) and Table 1, page 306-307}, which references~\cite{EGA4-3}*{11.2.6(ii)}, implies that we may assume that \(\widehat \Psi\) and \(\Psi\) are flat (again, at the expense of possibly enlarging \(T\)). There are only finitely many places of \(k\) lying over the primes bounded by \(\deg(\widehat\psi)\), so after adding these finitely many places to \(T\), we have verified Definition~\ref{def:GoodReductionOfCurves}\eqref{gr:smoothmodel} and~\eqref{gr:flatmorphism}.

Now we turn to verifying Definition~\ref{def:GoodReductionOfCurves}\eqref{gr:etaleramification}. Since there is no nontrivial \'etale cover of $\PP^1_k$ and \(C\) is geometrically integral, the morphism \(\psi\) cannot be \'etale, but it is generically \'etale since \(\Char(k) = 0\). Thus, there is a minimal proper smooth closed subscheme \(Z\subset C\) and \(\widehat Z\subset \widehat C\) such that \(\psi\) and \(\widehat\psi\) are \'etale away from \(Z\) and \(\widehat Z\), respectively. Let \(\calZ\subset \calC\) and \(\widehat\calZ\subset \widehat\calC\) be the Zariski closure of \(Z\) and \(\widehat Z\) respectively. Since each map in \(\calZ \to \Psi(\calZ)\to\Spec \OO_{k,T}\) and \(\widehat\calZ \to \Psi(\widehat\calZ)\to\Spec \OO_{k,T}\) is generically \'etale, as above we may use~\cite{PoonenRaPt}*{Theorem 3.2.1(ii) and Table 1, page 306-307}, which references~\cite{EGA4-4}*{17.7.8(ii)}, to assume (after possibly enlarging \(T\)) that each of the four morphisms \(\calZ \to \Psi(\calZ)\to\Spec \OO_{k,T}\) and \(\widehat\calZ \to  \Psi(\widehat\calZ)\to\Spec \OO_{k,T}\) are \'etale.

It remains to prove the claim about the Galois group. Let \(v\in \Omega_k\smallsetminus T\). We have a finite morphism \(\widehat\Psi \colon \widehat\calC \to \PP^1_{\OO_v}\) whose extension of function fields is Galois. Since the special fiber is geometrically irreducible, the decomposition group of the special fiber is equal to \(\Gal(\widehat{C}/\PP^1_{k_v})\) and similarly, the decomposition group of the special fiber of \(\widehat  \calC \to \calC\) is equal to \(\Gal(\widehat{C}/C)\). Thus, the special fiber of \(\widehat{\Psi}\) is Galois and the Galois group of the special fiber of \(\widehat \calC \to \calC\) contains no nontrivial normal subgroup. Hence, the special fiber of \(\widehat{\Psi}\) is the Galois closure of the special fiber of \(\Psi\), as desired.
\end{proof}

\begin{lemma}[An extension of Hilbert irreducibility]\label{lem:HilbertIrred}
    Let \(S\subset \Omega_k\) be a finite set of places, let \(\psi\colon C\to \PP^1_k\) be a nonconstant morphism of smooth projective geometrically integral curves, and fix a \(t_v\in \PP^1(k_v)\) for each \(v\in S\). 
 Then there exists a \(t\in \PP^1(k)\) such that \(C_t\) is irreducible (and hence \(\kk(C_t)\) is a field) and such that \(\kk(C_t)\otimes_k k_v \simeq \kk(C_{t_v})\) for all \(v\in S\). 
\end{lemma}
\begin{proof}
    Hilbert's irreducibility theorem implies that we can take $t'\in\PP^1(k)$ such that $C_{t'}$ is irreducible. Set \(L := \kk(C_{t'})\), set \(\widehat{L}\) to be the Galois closure of \(L/k\), and set \(G := \Gal(\widehat{L}/k)\).  Fix generators \(g_1, \dots, g_r\in G\). By the Chebotarev density theorem, there are places \(v_1, \dots, v_r\in \Omega_k\smallsetminus S\) such that the Frobenius element of \(w_i|v_i\) in $G$ generates the same subgroup as \(g_i\). By Krasner's Lemma, there exists an open neighborhood \(t'\in U_{v_i}\subset \PP^1(k_{v_i})\) such that \(\kk(C_t)\otimes_k k_{v_i} \simeq \kk(C_{t'})\otimes_k k_{v_i}\) for all \(t\in U_{v_i}\cap \PP^1(k)\)~\cite{PoonenRaPt}*{Prop. 3.5.74}. In particular, since \(g_i\) were chosen to generate \(G\), if \(t\in \PP^1(k) \cap \prod_i U_{v_i}\), then \(\kk(C_t)\) is a field and the Galois group of its Galois closure contains \(G\). Furthermore, for \(v\in S\), Krasner's Lemma implies that there exists an open neighborhood \(t_v\in U_v\subset \PP^1(k_v)\) such that \(\kk(C_t)\otimes_k k_v \simeq \kk(C_{t_v})\) for all \(t\in U_v\cap \PP^1(k)\). Then weak approximation allows us to find $t\in\PP^1(k)\cap \left(\prod_{v\in S} U_v \times \prod_i U_{v_i} \times \prod_{v\notin S\cup T}\PP^1(k_v)\right)$ and thus $C_t$ has the claimed properties.
\end{proof}

\subsection{Uniformity in realizable unramified local splitting behavior}\label{sec:ChebotarevGlobal}

\begin{theorem}[c.f.~\cite{DG2012}*{Theorem 1.2}]\label{thm:ChebotarevMain}
    Let \(C\) be a smooth projective geometrically integral curve over a number field \(k\), let \(\psi\colon C \to \PP^1_k\) be a nonconstant morphism, and let \(B\) denote the branch locus.  Assume that the Galois closure of \(\psi\colon C \to \PP^1\) is geometrically integral.
    Then for all nonarchimedean \(v\in \Omega_k\) with sufficiently large residue field, the isomorphism types of unramified residue fields is completely determined by the Galois group \(G\) of the Galois closure of \(\psi\).  More precisely, there is a bijection 
    \begin{align*}
       \left\{\ct(g) : g \in G\right\} \quad &\leftrightarrow\quad
        \left\{\kk(C_t)\otimes_k k_v : t\in \PP^1(k), t\bmod v \notin B(\F_v)\right\}/\sim,\\
        (d_i)_{i=1}^r \quad \quad \quad&\mapsto \quad\prod_{i=1}^rk_v(d_i).
        \end{align*}   
\end{theorem}
\begin{remarks}\hfill
\begin{enumerate}
    \item This theorem is essentially~\cite{DG2012}*{Theorem 1.2} combined with Proposition~\ref{prop:etalealg-nonGal}. However, as the proof is fairly short given everything we have proved thus far and citing~\cite{DG2012}*{Theorem 1.2} directly would require verifying that our assumptions imply their assumptions, we elect to give a direct proof.
    \item If the Galois closure of \(\psi\colon C \to \PP^1_k\) gives a smooth projective curve \(\widehat{C}\) that is not geometrically integral, then \(\widehat{C}_{k_v}\) may be reducible and the Galois closure of \(\psi\colon C_{k_v} \to \PP^1_{k_v}\) may be a proper subcover of \(\widehat{C}_{k_v}\). In this case, one must instead consider the Galois group of the special fiber. Since \(\overline{k}\cap \kk(\widehat{C})\) is a finite extension of \(k\), there are only finitely many possibilities for the Galois group of the special fiber, and one can use the proof of Theorem~\ref{thm:ChebotarevMain} for each of these possibilities to obtain an analogous result.
    \item The required lower bound on the size of the residue field comes from the effective Chebotarev Density Theorem (see Proposition~\ref{prop:EffCheb}\eqref{part:EffCheb}). As such, it can be made effective, and the explicit bound can be deduced from the proof. 
    \item The bijection in the theorem can be used both ways. For a single place with sufficiently large (but not too large) residue field, one can directly compute the set on the right hand side, from which one can deduce the set on the left hand side. Then one can use the bijection in the reverse direction to deduce splitting field behavior for all places with sufficiently large residue fields.
    \end{enumerate}
\end{remarks}

\begin{proof}
    Let $\widehat{\psi}\colon \widehat{C} \to \PP^1$ be the Galois closure of $\psi$, $G$ denote $\Gal(\widehat{C}/\PP^1) \leq S_d$, and $\widehat{g}$ denote the genus of $\widehat{C}$. By Lemma~\ref{lem:GoodReduction}, there exists a finite set of places $T \subset \Omega_k$ and morphisms $\widehat{\calC} \to \calC \to \PP^1_{\calO_{k,T}}$ of schemes that are smooth over $\calO_{k,T}$ such that the special fibers $\widehat{\calC}_k$ and $\calC_k$ are isomorphic to $\widehat{C}$ and $C$ respectively; and moreover, for all nonarchimedean $v \not \in T$, the morphisms $\Psi\colon\calC \to \PP^1$ and $\widehat{\Psi}\colon\widehat{\calC} \to \PP^1$ both have good reduction at $v$. By possibly enlarging $T$, we may assume that the \(O\) term in Proposition~\ref{prop:EffCheb}\eqref{part:EffCheb} is small enough to guarantee the existence of
    a point $\overline{t}_v \in (\PP^1\smallsetminus B)(\F_v)$ such that 
    \[
    \kk(\calC_{\overline{t}_v})\simeq \prod_{i=1}^r \F_v(d_i), \quad \text{where }\ct(c) = (d_i)_{i=1}^r
    \]
    for all conjugacy classes \(c\) of elements in \(G\) and all \(v\notin T\).\footnote{In fact, using the full statement of the Effective Chebotarev Density Theorem~\cite{Murty-Scherk}*{Theorem 1}, one can show that taking \(T\) so that \(q_v > \deg(B) - 1 + 4[(\widehat{g}+|G|)\sqrt{q_v} + |G|(\widehat{g}-1+|G|)]\) for all \(v\notin T\) is sufficient.} So, by Proposition~\ref{prop:EtAlgUnramified}, there exists a $t_v \in \PP^1(k_v)$ such that $\kk(\calC_{t_v}) \simeq \prod_{i=1}^r k_v(d_i)$, so the map in the theorem is well-defined by Lemma~\ref{lem:HilbertIrred}.

    On the other hand, if $t\in \PP^1(k)$ such that $t \bmod v \not\in B(\F_v)$, then $t\bmod v$ is not in the branch locus of $\Psi_{\F_v}$, so by Propositions~\ref{prop:EtAlgUnramified} and \ref{prop:etalealg-nonGal}, $\kk(C_t) \otimes_k k_v\simeq \prod_{i=1}^r k_v(d_i)$ where $(d_i)_{i=1}^r$ is the cycle type of the Frobenius conjugacy class at $t\bmod v$, which is a conjugacy class of $G$. So we see that the map in the theorem is a bijection.
\end{proof}

\subsection{Proof of Theorem~\ref{thm:SdIntro}}\label{sec:PfSdIntro} Let \(\widehat \psi\colon \widehat C\to \PP^1\) be the Galois closure of \(\psi\) and let \(G=\Gal(\widehat{C}/\PP^1) < S_d\). By assumption on \(\psi\), the Galois group of the Galois closure of \(\psi_{\kbar}\colon C_{\kbar}\to \PP^1_{\kbar}\) is generated by transpositions and is transitive, so must be \(S_d\) (see Lemma~\ref{lem:TransitiveTranspositions} below for details). Furthermore, \(G\) must contain the geometric Galois group of \(G\) must also be \(S_d\). In particular, \(\widehat{C}\) is geometrically integral, and therefore \eqref{part:InertiaSd} follows by Theorem~\ref{thm:ChebotarevMain}.

    By Lemma~\ref{lem:GoodReduction}, there exists a finite set of places \(T\subset \Omega_k\) such that for all nonarchimedean \(v\notin T\), the morphisms \(C\to \PP^1_k\) and \(\widehat{C}\to \PP^1_k\) have good reduction. Every nonarchimedean place with sufficiently large residue field is outside \(T\), thus we may assume that \(\psi\) has good reduction at \(v\). Let $\Psi\colon \calC \to \PP^1_{\OO_v}$ be an integral model as in Definition~\ref{def:GoodReductionOfCurves}.

    \textbf{\eqref{part:RamificationUpperBoundSd}} Let \(t\in \PP^1(k)\) be such that \(\kk(C_t)\) is ramified at \(v\), and let \(\overline{t}\in \PP^1(\F_v)\) be the specialization of \(t_v\). Then by Theorem~\ref{thm:RamAndInertiaDegrees}, any maximal subfield \(L\) of \(\kk(\calC_{t_v})\) has ramification degree at most \(e(\overline{x}/\overline{t})\) for some \(\overline{x}\in \calC_{\overline{t}}\). By assumption, \(e(\overline{x}/\overline{t}) \leq 2\), so \(e(L/k_v)\leq 2\) for all maximal subfields \(L\subset \kk(C_t)\otimes_k k_v\). 
    It remains to show that the number of ramified places is bounded above by 
        \[
            \max\left\{ \#\{x \in C_t : e(x/t) > 1\} : t\in \PP^1(k_v)\ \mbox{a branch point}\right\}.
        \]
Recall that the \'etale algebra \(\kk(C_t)\otimes_k k_v\) decomposes as a product:
    \[
    \kk(C_t)\otimes_k k_v= \kk(\calC_{t_v})=\prod_{\overline{x}\in \calC_{\overline{t}}} \kk(\calC_{t_v})_{\overline{x}}, \quad\textup{where}\quad \kk(\calC_{t_v})_{\overline{x}} := \prod_{\substack{x\in (\calC_{t_v})_{k_v}\\x\rightsquigarrow \overline{x}}}\kk(x).
    \]
    Theorem~\ref{thm:RamAndInertiaDegrees}\eqref{part:Ramification} says that any maximal subfield \(L\) of \(\kk(\calC_{t_v})_{\overline{x}}\) has ramification index bounded by \(e(\overline{x}/\overline{t})\), and furthermore if the ramification index of \(L/k_v\) is equal to \(e(\overline{x}/\overline{t})\), then \(L = \kk(\calC_{t_v})_{\overline{x}}\). The assumptions of Theorem~\ref{thm:SdIntro} imply that \(e(\overline{x}/\overline{t})\leq 2\). Thus if \(w|v\) is ramified in \(\kk(C_t)\) then \(\kk(C_t)_w \simeq \kk(\calC_{t_v})_{\overline{x}}\) for a unique \(\overline{x}\) with \(e(\overline{x}/\overline{t}) = 2\).
    
    \textbf{\eqref{part:RamificationExistenceSd}} Assume that \(\psi\) has a \(k_v\)-rational branch point \(t_v\in \PP^1(k_v)\), that there exists an odd degree ramification point \(x_v\in C\) lying above \(t_v\), and let \(L/\kk(x_v)\) be a quadratic ramified extension. By the assumptions on \(\psi\), \(e(x_v/t_v)=2\) and since \(\psi\) has good reduction the same is true for the ramification index and degree of the specializations \(\overline{x},\overline{t}\) of \(x_v, t_v\), respectively. Since \(2=e(\overline{x}/\overline{t})\) is relatively prime to \(\deg(x_v) = \deg(\overline{x})\), Theorem~\ref{thm:RamAndInertiaDegrees}\eqref{part:CurveExistence} implies that there exists a \(t'_v\in \PP^1(k_v)\) such that \(\kk(\calC_{t'_v})_{\overline{x}}\simeq L\). Thus Lemma~\ref{lem:HilbertIrred} applied to \(S = \{v\}\) gives the desired result.
    \qed
    
    \begin{lemma}\label{lem:TransitiveTranspositions}
        The group \(S_d\) has no proper transitive subgroups generated by transpositions.
    \end{lemma}
    \begin{proof}
        We prove by induction on $d$. The case when $d=1$ is trivial.
        Let $H$ be a transitive subgroup of $S_d$ generated by transpositions. Let $a,b \in \{1, \ldots, d-1\}$. Because $H$ is transitive and generated by transpositions, $H$ contains a sequence of transpositions $(a,i_1), (i_1, i_2), \ldots, (i_{n-1}, b)$, and hence we have $(a,b)\in H$. So $\Stab_H(d)$ is generated by transpositions of the set $\{1, \ldots, d-1\}$ by our assumption, and acts transitively on \(\{1,..., d-1\}\). Thus, by the induction hypothesis, $\Stab_H(d) = \Stab_{S_d}(d)\simeq S_{d-1}$. Then we have $H=S_d$, since $H$ must contain a transposition moving $d$.    
    \end{proof}

    \subsection{Modulo \(m\) reductions of points of bounded height}\label{sec:ModM}

        We restrict to the case \(k = \Q\). The isomorphism class of \(\Q_p\)-algebras \(\kk(C_t)\otimes_{\Q} \Q_p\) is controlled by the image of \(t\in \PP^1(\Z/p^r\Z)\) for some \(r\). Thus, it is natural to ask how points of bounded height distribute over the congruence classes modulo \(p^r\), or, more generally, modulo \(m\). 
        
        In this section, we modify standard number theoretic arguments to answer this question. In Lemma~\ref{lem:LowHeight}, we determine the minimal \(N\) such that every modulo \(m\) congruence class can be realized by a point of height at most \(N\). In Proposition~\ref{prop:equidistribution}, we show that as the height grows, the points of bounded height equidistribute over the congruence classes.
        
        \begin{lemma}\label{lem:LowHeight}
        For any positive integer \(m\), the map
            \[
                \{x \in \PP^1(\Q) : H(x) \leq N\} \to \PP^1(\Z/m\Z)
            \]
            is surjective if \(N\geq \max\left\{\lfloor\sqrt{m}\rfloor,\ \frac{m}{\min\{p\text{ prime}\,:\,p\mid m\}}\right\}\) and injective if \(N<\sqrt{m/2}\).  Moreover, these bounds are sharp: 
            \begin{enumerate}
                \item If \(p = n^2 + 1\) is prime, then \([n:1]\in \PP^1(\F_p)\) is not the image of any point \(x\in \PP^1(\Q)\) with \(H(x) < n = \lfloor \sqrt{p}\rfloor\).\label{eq:sharp-surj-prime}
                \item If \(m = 2m_0\), then \([m_0 : 1]\in \PP^1(\Z/m\Z)\) is not the image of any point \(x\in \PP^1(\Q)\) with \(H(x) < m_0 = \frac{m}{2}\).\label{eq:sharp-surj-composite}
                \item If \(p = (n-1)^2 + n^2\) is prime, then \(x := [n:1-n]\) and \(x' := [n-1:n]\) have height \(n = \lceil\sqrt{p/2}\rceil\) and \(x\equiv x'\bmod p\).\label{eq:sharp-inj}
            \end{enumerate}
        \end{lemma}
        \begin{proof}
        Surjectivity: The points \(0, \infty\in \PP^1(\Z/m\Z)\) are evidently the reduction of \(0\) and \(\infty\) in \(\PP^1(\Z)\).
        Given a point in \(\PP^1(\Z/m\Z)\smallsetminus \{0,\infty\}\), we may find a representative $[uw:v]$, where $u,v,w\in\{1,\dots,m-1\}$, $u,v\mid m$ and $\gcd(u,v)=\gcd(w,m)=1$. By symmetry, we may assume that \(v<u\), which implies that \(v<\sqrt{m}\).
        Consider the lattice $L:=\Z\cdot(uw,v) + (m\Z)^2\subset \Z^2$. 
        Since $\frac{m}{v}(uw,v)-(0,m)=(\frac{m}{v}uw,0)$ and $\gcd(\frac{m}{v}u,m)=\frac{m}{v}\gcd(u,v)=\frac{m}{v}$, the lattice $L$ is generated by $(uw,v)$ and $(\frac{m}{v},0)$, and so $\det L=m$.
        
        For \(\lambda\in \RR_{>0}\), Minkowski’s convex body theorem \cite{CasselsGoN}*{p.\,71, Theorem II} implies that the region 
        \[\mathcal{R}_{\lambda}=\left\{(a,b)\in\RR^2:|a|\leq \frac{m}{v\lambda},\ |b|\leq v\lambda\right\}\]
        contains a non-zero vector $({s_{\lambda},t_{\lambda}})\in L$. 

        Since \(\det L=m\), we may assume that \(\gcd(s_{\lambda}, t_{\lambda}) \mid m\). We will show that if \(1<\lambda<\min\{\frac{\sqrt{m}}{v},\ p:p\mid \frac{m}{v}\}\), then \(\gcd(t_{\lambda}, m) = v\) and \(\gcd(s_{\lambda}, t_{\lambda}) = 1\).
        By the definition of $L$, we have $s_{\lambda}\equiv guw\pmod m$  and $t_{\lambda}\equiv gv\pmod m$ for some integer $g$.
        In particular, $v\mid t_{\lambda}$. Further, since $|t_{\lambda}|\leq v\lambda<v\min\{ p:p\mid \frac{m}{v}\}$, we deduce that either $t_{\lambda}=0$ or $\gcd(t_{\lambda},m)=v$, in which case $\gcd(g,m)=1$ and \(\gcd(s_{\lambda}, t_{\lambda}, m) = \gcd(uw, v) = 1\). If $t_{\lambda}=0$ then $\frac{m}{v}$ must divide $g$ and hence also $s_{\lambda}$, but $|s_{\lambda}|\leq\frac{m}{v\lambda} < \frac{m}{v}$ forces $s_{\lambda}=t_{\lambda}=0$, a contradiction.
        
        The assumption $\lambda<\frac{\sqrt{m}}{v}$ further implies that $\frac{m}{v\lambda}>v\lambda$, so $\max\{|s_{\lambda}|,|t_{\lambda}|\}\leq \max\{\frac{m}{\lambda v}, v\lambda\} = \frac{m}{v\lambda}$. Thus, as \(\lambda\) increases towards this bound, the height of \((s_{\lambda}, t_{\lambda})\) is bounded above by \(\frac{m}{v\lambda}\) which decreases to \(\max\{\sqrt{m}, \frac{m}{v\min\{p: p |\frac{m}{v}\}}\}\). Since there can only be finitely many points on the lattice with bounded height, 
        this limiting argument yields
        a non-zero vector $(s,t)\in L$ that satisfies
       \[\max\{|s|,|t|\}\leq \max\left\{\sqrt{m},\ \frac{m}{v\min\{p:p\mid \frac{m}{v}\}}\right\}\leq \max\left\{\sqrt{m},\ \frac{m}{\min\{p:p\mid m\}}\right\}.
       \]
       Since $\max\{|s|,|t|\}\in\Z$, we can replace the upper bound by its floor.
        
        Injectivity:
        Suppose $x,x'\in \PP^1(\Q)$, satisfying $H(x),H(x')<\sqrt{m/2}$, reduce to the same point on $\PP^1(\Z/m\Z)$ that is not $0$ or $\infty$. Write $x=[a:b]$ and $x'=[c:d]$, where $a,b,c,d\in\Z$ all have absolute value less than $\sqrt{m/2}$, satisfy $\gcd(a,b)=\gcd(c,d)=1$ and $b,d\geq 1$, and such that  $[a:b]\equiv [c:d]\bmod m$.
        Then $m\mid ad-bc$, but by the height bound, $|ad-bc|\leq |ad|+|bc|<m$. This forces $ad-bc=0$. The coprimality condition implies that $(a,b)=\pm (c,d)$. Since $b,d$ are both positive, it must be that $(a,b)=(c,d)$, hence $x=x'$.

        {Sharpness: The claims in~\eqref{eq:sharp-surj-composite} and~\eqref{eq:sharp-inj} are straightforward.} It remains to show~\eqref{eq:sharp-surj-prime}. If \(a,b\in \Z\) are integers with absolute value at most \(n-1\), then \(|nb - a| \leq n(n-1) + (n-1) = n^2 - 1\).  So if \(p = n^2 + 1\) is a prime, then \(\det\left(\begin{bmatrix}n& 1\\a &b \end{bmatrix}\right) = nb - a\) is nonzero if \(ab\neq 0\). Thus, there is no point with height at most \(n-1\) that maps to the point \([n:1]\in \PP^1(\F_p)\).
        \end{proof}

        \begin{prop}\label{prop:equidistribution}
            Fix a positive integer $m$ and a point \({[u:v]}\in \PP^1(\Z/m\Z)\). Then
            \[
           \#\{x\in \PP^1(\Q) : H(x)\leq N, \; x\bmod m = {[u:v]}\} =\frac{12N^2}{\pi^2 m}\prod_{p\mid m}\frac{p}{p+1}+O\left(N\log N\right).
            \]
        \end{prop}
        \begin{proof}
        We follow the standard argument of counting points of bounded height, first written out by Schanuel, while keeping track of the modulo \(m\) images.
        
        We would like to estimate
        \[F(N)=\#\{(a,b)\in\Z^2: b\geq 1,\ a\neq 0,\ \max\{|a|,|b|\}\leq N,\ \gcd(a,b)=1,\ [a:b]\bmod m={[u:v]}\}.\]
        The only points that have been dropped from  $\PP^1(\Q)$ are $[0:1]$ and $[1:0]$, which are finitely many points. 
        
        We first allow $\gcd(a,b)\geq 1$ but assume that $\gcd(a,b,m)=1$, and will fix this discrepancy later.
        The condition $[a:b]\equiv [u:v]\bmod m$ is equivalent to $a\equiv ug\bmod m$ and $b\equiv vg\bmod m$ for some integer $g\in(\Z/m\Z)^\times$. There are precisely $\varphi(m)=\#(\Z/m\Z)^\times$ many choices of  $g\in(\Z/m\Z)^\times$ satisfying the equations because \(u,v\) generate \(\Z/m\Z\).
        We cut up the box defined by $1\leq b\leq N$ and $1\leq |a|\leq N$ into $m\times m$ boxes, we have
        \[\#\{(a,b)\in\Z^2:1\leq |a|, b\leq N,
        \ \gcd(a,b,m)=1,\ [a:b]\equiv {[u:v]}\bmod m\}
                =\varphi(m)\frac{2N^2}{m^2}+O(N).
        \]
        
        To impose the $\gcd(a,b)=1$ condition, we will carry out M\"obius inversion.
        Suppose $A\mid \gcd(a,b)$, then we can take a factor of $A$ out of both $a$ and $b$, and the new $a,b$ runs up to $N/A$. Therefore we deduce that
        \[S(N,A){\colonequals}\sum_{1\leq b\leq N}\sum_{\substack{1\leq |a|\leq N\\  [a:b]\equiv [u:v]\bmod m\\ \gcd(a,b,m)=1\\ A\mid \gcd(a,b)}}1
        =\varphi(m)\frac{2N^2}{A^2m^2}+O\left(\frac{N}{A}\right).\]
        M\"obius inversion formula gives
        \[F(N)=\sum_{\substack{A\leq N\\\gcd(A,m)=1}} \mu(A)S(N,A)
            =\left(\frac{2N^2\varphi(m)}{m^2}\sum_{\substack{A\leq N\\\gcd(A,m)=1}}\frac{\mu(A)}{A^2}\right)+O\left(N\log N\right)
            .\]
        If we extend the inner sum to all positive integers $A$, we have 
        \[\sum_{\substack{A\\\gcd(A,m)=1}}\frac{\mu(A)}{A^2}
        =\prod_{p\nmid m}\left(1-\frac{1}{p^2}\right)
        =\frac{1}{\zeta(2)}\prod_{p\mid m}\left(1-\frac{1}{p^2}\right)^{-1}
        =\frac{6}{\pi^2}\prod_{p\mid m}\left(1-\frac{1}{p^2}\right)^{-1}.
        \]
        Truncating the infinite sum, we have
        \[\sum_{\substack{A\leq N\\\gcd(A,m)=1}}\frac{\mu(A)}{A^2}=\sum_{\substack{A\\\gcd(A,m)=1}}\frac{\mu(A)}{A^2}+O\left(\sum_{A> N}\frac{1}{A^2}\right)
        =\frac{6}{\pi^2}\prod_{p\mid m}\left(1-\frac{1}{p^2}\right)^{-1}+O\left(\frac{1}{N}\right).
        \]
        Therefore
        \[F(N)=\frac{12N^2\varphi(m)}{\pi^2 m^2}\prod_{p\mid m}\left(1-\frac{1}{p^2}\right)^{-1}+O\left(N\log N\right)
            = \frac{12N^2}{\pi^2 m}\prod_{i=1}^r\left(1+\frac{1}{p_i}\right)^{-1}+O\left(N\log N\right)\]
        as required.
        \end{proof}



\begin{bibdiv}
\begin{biblist}

\bib{Beckmann}{article}{
      author={Beckmann, Sybilla},
       title={On extensions of number fields obtained by specializing branched
  coverings},
        date={1991},
        ISSN={0075-4102,1435-5345},
     journal={J. Reine Angew. Math.},
      volume={419},
       pages={27\ndash 53},
         url={https://doi.org/10.1515/crll.1991.419.27},
      review={\MR{1116916}},
}

\bib{BELOV}{article}{
      author={Bourdon, Abbey},
      author={Ejder, \"Ozlem},
      author={Liu, Yuan},
      author={Odumodu, Frances},
      author={Viray, Bianca},
       title={On the level of modular curves that give rise to isolated
  {$j$}-invariants},
        date={2019},
        ISSN={0001-8708,1090-2082},
     journal={Adv. Math.},
      volume={357},
       pages={106824, 33},
         url={https://doi.org/10.1016/j.aim.2019.106824},
      review={\MR{4016915}},
}

\bib{CasselsGoN}{book}{
      author={Cassels, J. W.~S.},
       title={An introduction to the geometry of numbers},
      series={Classics in Mathematics},
   publisher={Springer-Verlag, Berlin},
        date={1997},
        ISBN={3-540-61788-4},
        note={Corrected reprint of the 1971 edition},
      review={\MR{1434478}},
}

\bib{DG2012}{article}{
      author={D\`ebes, Pierre},
      author={Ghazi, Nour},
       title={Galois covers and the {H}ilbert-{G}runwald property},
        date={2012},
        ISSN={0373-0956,1777-5310},
     journal={Ann. Inst. Fourier (Grenoble)},
      volume={62},
      number={3},
       pages={989\ndash 1013},
         url={https://doi.org/10.5802/aif.2714},
      review={\MR{3013814}},
}

\bib{DL-TAMS}{article}{
      author={D\`ebes, Pierre},
      author={Legrand, Fran\c~cois},
       title={Specialization results in {G}alois theory},
        date={2013},
        ISSN={0002-9947,1088-6850},
     journal={Trans. Amer. Math. Soc.},
      volume={365},
      number={10},
       pages={5259\ndash 5275},
         url={https://doi.org/10.1090/S0002-9947-2013-05800-X},
      review={\MR{3074373}},
}

\bib{Grant-Leitzel}{article}{
      author={Grant, Ken},
      author={Leitzel, Joan},
       title={Norm limitation theorem of class field theory},
        date={1969},
        ISSN={0075-4102,1435-5345},
     journal={J. Reine Angew. Math.},
      volume={238},
       pages={105\ndash 111},
      review={\MR{249400}},
}

\bib{EGA4-3}{article}{
      author={Grothendieck, A.},
       title={{\'E}l\'ements de g\'eom\'etrie alg\'ebrique. {IV}. {\'e}tude
  locale des sch\'emas et des morphismes de sch\'emas. {III}},
        date={1966},
        ISSN={0073-8301,1618-1913},
     journal={Inst. Hautes \'Etudes Sci. Publ. Math.},
      number={28},
       pages={255},
         url={http://www.numdam.org/item?id=PMIHES_1966__28__255_0},
      review={\MR{217086}},
}

\bib{EGA4-4}{article}{
      author={Grothendieck, A.},
       title={{\'E}l\'ements de g\'eom\'etrie alg\'ebrique. {IV}. {\'e}tude
  locale des sch\'emas et des morphismes de sch\'emas {IV}},
        date={1967},
        ISSN={0073-8301,1618-1913},
     journal={Inst. Hautes \'Etudes Sci. Publ. Math.},
      number={32},
       pages={361},
         url={http://www.numdam.org/item?id=PMIHES_1967__32__361_0},
      review={\MR{238860}},
}

\bib{Iwasawa55}{article}{
      author={Iwasawa, Kenkichi},
       title={On {G}alois groups of local fields},
        date={1955},
        ISSN={0002-9947},
     journal={Trans. Amer. Math. Soc.},
      volume={80},
       pages={448\ndash 469},
         url={https://doi.org/10.2307/1992998},
      review={\MR{75239}},
}

\bib{KLN-GrunwaldRamified}{article}{
      author={K\"onig, Joachim},
      author={Legrand, Fran\c{c}ois},
      author={Neftin, Danny},
       title={On the local behavior of specializations of function field
  extensions},
        date={2019},
        ISSN={1073-7928,1687-0247},
     journal={Int. Math. Res. Not. IMRN},
      number={9},
       pages={2951\ndash 2980},
         url={https://doi.org/10.1093/imrn/rny016},
      review={\MR{3947643}},
}

\bib{Murty-Scherk}{article}{
      author={Kumar~Murty, Vijaya},
      author={Scherk, John},
       title={Effective versions of the {C}hebotarev density theorem for
  function fields},
        date={1994},
        ISSN={0764-4442},
     journal={C. R. Acad. Sci. Paris S\'{e}r. I Math.},
      volume={319},
      number={6},
       pages={523\ndash 528},
}

\bib{Legrand}{article}{
      author={Legrand, Fran\c~cois},
       title={Hilbert specialization results with local conditions},
        date={2017},
        ISSN={0035-7596,1945-3795},
     journal={Rocky Mountain J. Math.},
      volume={47},
      number={6},
       pages={1917\ndash 1945},
         url={https://doi.org/10.1216/RMJ-2017-47-6-1917},
      review={\MR{3725250}},
}

\bib{QingLiu}{book}{
      author={Liu, Qing},
       title={Algebraic geometry and arithmetic curves},
      series={Oxford Graduate Texts in Mathematics},
   publisher={Oxford University Press, Oxford},
        date={2002},
      volume={6},
        ISBN={0-19-850284-2},
        note={Translated from the French by Reinie Ern\'{e}, Oxford Science
  Publications},
}

\bib{Neukirch}{book}{
      author={Neukirch, J\"{u}rgen},
       title={Algebraic number theory},
      series={Grundlehren der mathematischen Wissenschaften [Fundamental
  Principles of Mathematical Sciences]},
   publisher={Springer-Verlag, Berlin},
        date={1999},
      volume={322},
        ISBN={3-540-65399-6},
         url={https://doi.org/10.1007/978-3-662-03983-0},
        note={Translated from the 1992 German original and with a note by
  Norbert Schappacher, With a foreword by G. Harder},
}

\bib{PoonenRaPt}{book}{
      author={Poonen, Bjorn},
       title={Rational points on varieties},
      series={Graduate Studies in Mathematics},
   publisher={American Mathematical Society, Providence, RI},
        date={2017},
      volume={186},
        ISBN={978-1-4704-3773-2},
         url={https://doi.org/10.1090/gsm/186},
      review={\MR{3729254}},
}

\bib{VirayVogt}{article}{
      author={Viray, Bianca},
      author={Vogt, Isabel},
       title={Isolated and parameterized points on curves},
       pages={26pp.},
        note={Available at \url{https://arxiv.org/abs/2406.14353}},
}

\end{biblist}
\end{bibdiv}
\end{document}